\documentclass[10pt]{article}
\usepackage{amsmath, a4, latexsym,amssymb,bm}
\usepackage{enumerate}
\usepackage{datetime}
\usepackage{color}
\def\qed{\mbox{$\Box$}}
\def\boxp{\mbox{$\Box$}}
\newenvironment{proof}{\par\noindent{\bf Proof.\ }}{\hfill\qed\bigskip\par}
\newenvironment{proofn}{\par\noindent{\bf Proof.\ }}{\par}

\renewcommand{\marginpar}[1]{}
\def\FF{\mathbb{F}}
\def\CM{{\mathcal{M}}}
\def\CG{{\mathcal{G}}}

\def\CF{{\mathcal{F}}}
\def\CN{{\mathcal{N}}}
\title{Infinite rate symbiotic branching on the real line:\\ The tired frogs model}
\author{}
\author{
Achim Klenke\footnote{Both authors were supported partly by the German Israeli foundation, G.I.F. Grant 1170-186.6/2011.}
\\Institut f\"{u}r Mathematik
\\Johannes Gutenberg-Universit\"{a}t Mainz
\\Staudingerweg 9
\\D-55099 Mainz
\\Germany
\\math@aklenke.de
\and
Leonid Mytnik\footnote{Part of this work was done while L. Mytnik was in Mainz with a Humboldt research award.}
\\Faculty of Industrial Engineering\\ and Management
\\Technion -- Israel Institute of Technology
\\Haifa 32000\\ Israel
\\leonid@ie.technion.ac.il}
\date{23.10.2017\\revised version of 14.03.2019}
\newfam\msbmfam\font\tenmsbm=msbm10\textfont
\msbmfam=\tenmsbm\font\sevenmsbm=msbm7

\scriptfont\msbmfam=\sevenmsbm

\newtheorem{theorem}{Theorem}[section]
\newtheorem{lemma}[theorem]{Lemma}

\newtheorem{remark}[theorem]{Remark}
\newtheorem{proposition}[theorem]{Proposition}

\newtheorem{corollary}[theorem]{Corollary}

\newtheorem{assumption}[theorem]{Assumption}
\newcommand{\gdm}{\hfill\qed\bigskip\par}
\def\eope{\tag*{\boxp}}

\newcommand{\rest}[1]{\raisebox{-0.3em}{$\big|_{\scriptstyle
#1}$}}

\def\nn{\nonumber}
\usepackage{bm} \bmdefine\bDelta{\Delta}
\def\given{\hspace{0.8pt}|\hspace{0.8pt}}
\def\Given{\hspace{1.0pt}\big|\hspace{1.0pt}}

\def\sdv{\partial^2_x}

\def\EE{\mathbb{E}}

\def\cF{{\mathcal{F}}}

\def\cM{{\mathcal{M}}}\def\cN{{\mathcal{N}}}

\def\ep{\varepsilon}
\def\ve{\varepsilon}

\renewcommand{\cases}[1]{\left\{\begin{array}{rl}#1\end{array}\right.}
\newcommand{\equ}[1]{(\ref{#1})}
\newcommand{\mbu}{\quad\mbox{ and }\quad}
\newcommand{\mbs}[1]{\mbox{ \;#1\; }}
\newcommand{\mbsl}[1]{\mbox{ \;#1}}

\newcommand{\mas}{\quad\mbox{\;as \;}}
\newcommand{\mf}{\quad\mbox{\;for \;}}
\newcommand{\mfa}{\quad\mbox{\;for all \;}}

\newcommand{\mfasts}{\quad\mbox{almost surely}}

\newcommand{\N}{\mathbb{N}}
\newcommand{\R}{\mathbb{R}}

\newcommand{\To}[1]{\,\stackrel{#1}{\longrightarrow}\,}
\newcommand{\Toi}[1]{\To{#1 \rightarrow \infty}}

\newcommand{\limk}{\Toi{k}}

\renewcommand{\P}{\mathbf{P}}
\newcommand{\E}{\mathbf{E}}

\def\CA{\mathcal{A}}
\def\1{\mathbf{1}}

\newcommand{\ARG}{\,\boldsymbol{\cdot}\,}

\newcommand{\mfalls}{\quad\mbox{if \;}}
\newcommand{\msonst}{\quad\mbox{otherwise}}

\def\deltaone{\xi}
\def\roix{\frac12\partial^2_x}

\newcommand{\elll}[2]{\ell^{#1}_{#2}}
\newcommand{\elleta}[1]{\elll{\eta}{#1}}
\newcommand{\elle}[1]{\elll{}{#1}}
\newcommand{\ellk}[1]{\elll{k}{#1}}
\newcommand{\ellkm}[1]{\elll{k_m}{#1}}
\def\olt{\overline{\ell}_t}
\def\wtW{\widetilde W}
\begin{document}
\maketitle
\begin{abstract}
Consider a population of infinitesimally small frogs on the real line. Initially the frogs on the positive half-line are dormant while those on the negative half-line are awake and move according to the heat flow. At the interface, the incoming wake frogs try to wake up the dormant frogs and succeed with a probability proportional to their amount among the total amount of involved frogs at the specific site. Otherwise, the incoming frogs also fall asleep. This frog model is a special case of the infinite rate symbiotic branching process on the real line with different motion speeds for the two types.

We construct this frog model as the limit of approximating processes and compute the structure of jumps. We show that our frog model can be described by a stochastic partial differential equation on the real line with a jump type noise.
\end{abstract}\bigskip


\section{Introduction}
\label{S1}
\subsection{Motivation and First Results}
\label{S1.1}

Consider the following pair of stochastic partial differential equations
\begin{equation}
\label{E1.01}
\frac{\partial X^{\gamma,i}_t(x)}{\partial t}=\frac{c_i}{2}\partial^2_x X^{\gamma,i}_t(x) +\sqrt{\gamma X^{\gamma,1}_t(x)X^{\gamma,2}_t(x)}\,\dot{W^{i}}(t,x),\quad t\geq0,\,x\in\R.
\end{equation}
Here $\dot{W^{i}}$ is a space-time white noise, $i=1,2$, and $\dot{W}^1$ and $\dot{W}^2$ are correlated with parameter $\varrho\in[-1,1]$. The constants $c_1,c_2\geq0$ govern the speeds of dispersion. We interpret $X^{\gamma,i}_t(x)$  as the density of particles of type $i$ at site $x$ at time $t$. Each type of particles performs super-Brownian motion on the real line with local branching rate proportional to the density of the respective other type. This model hast been studied only for the case $c_1=c_2>0$: If $\varrho=0$, then the types are uncorrelated and this model was introduced by Dawson and Perkins \cite{DawsonPerkins1998} under the name \emph{mutually catalytic branching process}. For $\varrho\neq0$, this model was introduced by Etheridge and Fleischmann \cite{EtheridgeFleischmann2004} who coined the name \emph{symbiotic branching process}.

Recently, Blath et al \cite{BlathHammerOrtgiese2016} studied the limit of $\gamma\to\infty$ of this model for a range of negative values of $\varrho$. For $\varrho\geq0$, it is still an open problem how to construct the infinite rate limit of this model. If we replace the real line as site space by some discrete space and replace $\roix$ by the generator of some Markov chain on this site space, then for $\varrho=0$ the infinite rate process was studied in great detail in \cite{KM1}, \cite{KM2}, \cite{KM3} and \cite{KO}.

The main tool for showing (weak) uniqueness for the solutions of \eqref{E1.01} is a self-duality relation that goes back to Mytnik \cite{Mytnik1998} for the case $\varrho=0$ and Etheridge and Fleischmann \cite{EtheridgeFleischmann2004} in the case $\varrho\neq0$. Like many duality relations for genealogical or population dynamical models, the underlying principle of the duality is a back-tracing of ancestral lines. The viability of this method relies crucially on the fact, that the ancestral lines can be drawn without knowledge of the type of the individual. This is possible only in the absence of \emph{selection} and of a type-dependent motion. This is the deeper reason, why no simple duality relation could be established for the model \eqref{E1.01} in the case $c_1\neq c_2$.

Here, we make a step toward the model of infinite rate symbiotic branching with different speeds of motion by considering the extreme case $c_1=1$, $c_2=0$ and $\varrho=-1$. In other words, we consider a two-type model on the real line where only infinitesimal individuals of type 1 (which we imagine as green) move according to the heat flow while type 2 (red) stands still. Furthermore, by infinite rate branching, there cannot be both types present at any given site. Finally, at any given site, each population evolves in a martingale fashion while the sum of both types has no random drift since $\varrho=-1$ and hence the gains of type 1 are the losses of type 2 and vice versa.

Our model is a variation of a model often called \emph{frog model}. See, e.g., \cite{Popov2003}, \cite{GantertSchmidt2009}, \cite{KosyginaZerner2017}, \cite{FontesMachadoSarkar2004}. Loosely speaking, there are two kinds of particles distributed in space, named wake frogs (green) and sleeping frogs (red). Wake frogs move in space and activate sleeping frogs when they are in contact with them. Our model is different to the classical frog model mainly in the sense that wake frogs can activate sleeping frogs but also can get tired and fall asleep when they are in contact with dormant frogs. Furthermore, our frogs are of infinitesimal size and hence move deterministically according to the heat flow.

When a  (infinitesimal) wake frog of size $dx$ encounters a colony of sleeping frogs of size $r$ it either falls asleep (becomes red) or wakes the whole colony (turns them green). The latter happens with probability $dx/r$ which makes the number of dormant frogs a martingale.

For simplicity, let us explain the basic concepts in the discrete space situation first. Assume that $S$ is a countable site space and that $\CA$ is the generator of a continuous time Markov chain on $S$. Let $\CA^*$ be the adjoint of $\CA$ with respect to the counting measure on $S$. That is
$$\langle \CA^*f,g\rangle=\langle f,\CA g\rangle$$
for all suitable $f,g$ and where $\langle f,g\rangle=\sum_{i \in S}f(i)g(i)$.  Let $E:=[0,\infty)^2\setminus(0,\infty)^2$ and
$$\EE:=\Big\{x=(x^1,x^2)\in E^S:\,\sum_{k\in S} \big(x^1(k)+x^2(k)\big)<\infty\Big\}.$$
By a solution of the martingale problem MP$_S$\, we understand an  $\EE$-valued Markov process $(X^1,X^2)$ with c\`{a}dl\`{a}g paths such that
\begin{equation}
\label{E1.02}
\begin{aligned}
X^1_t(k)&:=X^1_0(k)+\int_0^t\CA^* X^1_s(k)\,ds + M_t(k)\\
X^2_t(k)&:=X^2_0(k) - M_t(k)
\end{aligned}
\end{equation}
for some orthogonal zero mean martingales $M(k)$, $k\in S$. As usual, uniqueness of the solution to a martingale problem means uniqueness in law.
\begin{theorem}
\label{T1}
If $S$ is finite, then for any initial condition $(X^1_0,X^2_0)\in\EE$, there exists a unique solution of MP$_S$.
\end{theorem}
\begin{remark}
\label{R1.2}
If in the martingale problem MP$_S$, we would allow local coexistence of types, that is $X^1_t(k)X^2_t(k)$ could be positive, and we define $M(k)$ to be a continuous martingale with square variation process $\langle M(k)\rangle_t=\int_0^t\gamma X^1_s(k)X^2_s(k)\,ds$, then we would have the process of finite rate $\gamma\geq0$ symbiotic branching with $\varrho=-1$. It is standard to show that if we let $\gamma\to\infty$, then we get convergence to the infinite rate model defined above. This programme has been carried out for similar models, e.g., in \cite{KM2} and \cite{DoeringMytnik2012}.
\end{remark}
\begin{remark}
\label{R1.3}
It is standard to extend the existence result in Theorem~\ref{T1} to countable sets $S$ under some mild regularity assumptions on $\CA$, e.g., for random walks on an abelian group $S$. This is done, for example, using an approximation scheme with finite subsets of $S$.
\end{remark}
\begin{remark}
\label{R1.4}
In order to stress the formal similarity with the corresponding processes on $\R$ instead of $S$, it is convenient to have a weak formulation of \eqref{E1.02}. Note that \eqref{E1.02} is equivalent to
\begin{equation}
\label{E1.03}
\begin{aligned}
\big\langle X^1_t,\phi_1\big\rangle&=\big\langle X^1_0,\phi_1\big\rangle+\int_0^t\big\langle X^1_s,\CA\phi_1\big\rangle\,ds + \langle M_t,\phi_1\rangle\\
\big\langle X^2_t,\phi_2\big\rangle&=\big\langle X^2_0,\phi_2\big\rangle - \langle M_t,\phi_2\rangle
\end{aligned}
\end{equation}
for finitely supported functions $\phi_1,\phi_2$. Here the martingales $\langle M_t,\phi_1\rangle$ and $\langle M_t,\phi_2\rangle$ are orthogonal for functions with disjoint supports.
\end{remark}

A preliminary step towards Theorem~\ref{T1} is the one-colony model. Here the single colony either hosts dormant frogs (type 2) or wake frogs (type 1). At varying speed $\theta_s$ at any time $s$ infinitesimal wake frogs arrive. They try to wake up the sleeping frogs and succeed with probability $(\theta_s/X^2_s)\,ds$ in the time interval $ds$. Otherwise they also fall asleep. After the awakening, the colony will host only wake frogs. The wake frogs still arrive at the varying speed $\theta_s$. In addition, they emigrate at a speed proportional to the number of wake frogs.

To be more formal, let $c\geq0$ be a constant determining the strength of emigration of type 1 and let $(\theta_s)_{s\geq0}$ be a nonnegative measurable deterministic map that governs the rate of immigration of type 1. We consider a Markov process $X=(X^1,X^2)$ with values in $E$ and with c\`{a}dl\`{a}g paths which solves the martingale problem MP$_0$\,:

\begin{equation}
\label{E1.04}
\begin{aligned}
X^1_t&:=X^1_0+\int_0^t \theta_s\,ds  - \int_0^t c X^1_s\,ds + M_t\\
X^2_t&:=X^2_0 - M_t
\end{aligned}
\end{equation}
for some zero mean martingale $M$.
\begin{theorem}
\label{T2}
For any initial condition $(X^1_0,X^2_0)\in E$, there exists a unique solution of MP$_0$.
\end{theorem}

Note that as long as $X^1_t=0$, we have $-M_t=\int_0^t \theta_s\,ds$ is the amount of additional frogs that have fallen asleep by time $t$. Thus $-M_t$ adds to the initially dormant frogs $X^2_0$ to give $X^2_t$. At some random time, all frogs wake up and $M_t$ jumps to the value $X^2_0$. This random time point is chosen such that $(M_t)$ is a martingale.

Before we proceed, let us heuristically describe the evolution of the processes solving martingale problems
MP$_0$ and MP$_S$. Let us start with the one-colony model, that is, with the process $(X^1,X^2)$ solving
MP$_0$. Since $X^2$ is a nonnegative martingale, if $X^2_t=0$, then we have $X^2_s=0$ for all $s\geq t$ (in fact, $\E[X_s\given X_t=0]=0$ due to the martingale property and $X_s\geq0$ a.s.{} which implies $X_s=0$ a.s.{} on the event $\{X_t=0\}$). Hence, the process is non-trivial only if $X^2_0>0$ and thus $X^1_0=0$.  Since $0$ is a trap for $X^2$ and $X_s\in E$, the martingale $X^2_t$ is continuous and monotonically increasing with derivative $\theta_t$ up to the random time $\tau$ where it has a single jump to $0$, that is, $X^2_s=0$ if and only if $s\geq \tau$. (Compare with the process $Z$ defined in Lemma~\ref{L2.1} below.) Moreover, as we see from the equations, at the same time $\tau$, $X^1$ makes a jump up, becomes positive and solves the deterministic equation $\partial_t X^1_t=\theta_t-cX^1_t$ for $t\geq\tau$.

Now we will give a  more detailed description of the jump time $\tau$, although most of the technical details will be provided   in the proof of Theorem~\ref{T2} in Section~\ref{S2}.
Let $W$ be a nonnegative random variable whose distribution is given by
\begin{equation}
\label{E1.05}
\P[W>x]=\frac{X^2_0}{X^2_0+x}\mfa x>0.
\end{equation}
That is, the hazard rate of $W$ is $x\mapsto 1/(X^2_0+x)$. (Recall that the hazard rate of a nonnegative random variable $W$ is the map $x\mapsto h(x):=\P[W\in dx\given W\geq x]/dx$, or in terms of its distribution function $F_W$: $h(x)=F'_W(x)/(1-F_W(x))$.)
If $\int_0^\infty\theta_sds=\infty$, then due to the martingale property, we have (see Proposition~\ref{P2.2} for a formal proof)
\begin{equation}
\label{E1.06}
\sup\{X^2_t:\,t\geq0\}\stackrel{d}=W+X^2_0.
\end{equation}
Hence, we can use $W$ to define
\begin{equation}
\label{E1.07}
\tau=\inf\left\{t\geq0:\,\int_0^t\theta_s\,ds>W\right\}.
\end{equation}
With this $\tau$ at hand, we can define $(X^1,X^2)$ as described above and get a solution of MP$_0$ even if $\int_0^\infty \theta_sds<\infty$.

An alternative point of view is the Markov structure of the process and we could construct $\tau$ in a more ``adapted'' fashion which will be useful if we pass to models with many colonies. Note that if $X^2_t>0$, then it increases with derivative $\theta_t\geq0$ until it jumps down to $0$. Very loosely speaking, the expectation of $X^2_t$ increases in the time interval $dt$ by $\theta_tdt$ due to the continuous growing and it decreases by $X^2_tdt$ times the jump rate. Hence, in order that $X^2$ be a martingale, the jump rate must be $\theta_t/X^2_t$. That is, the hazard rate of $\tau$ is $\theta_t/X^2_t$. Hence, we can define $\tau$ as the first time of a point of a Poisson process with the appropriate rate. We do so by letting $\CN(dt,dr)$ a Poisson point process on $\R_+^2$ with intensity measure
\begin{equation}
\label{E1.08}
\CN'(dt,dr)=dt\,dr.
\end{equation}
Then we define the jump rate $\displaystyle I_t:=\frac{\theta_t}{X^2_t}\1_{\{X^2_t>0\}}$ and let
$$\tau:=\inf\left\{t\geq0:\,\int_0^t\int_0^\infty\1_{[0,I_s]}(r)\CN(ds,dr)>0\right\}.$$
If we define $W:=\int_0^\tau\theta_t\,dt$, then it is easy to check that \eqref{E1.05} holds if $\int_0^\infty\theta_t\,dt=\infty$.
Finally, we define the martingale measure
\begin{equation}
\label{E1.09}
\CM:=\CN-\CN'.
\end{equation}

The advantage of this construction is that we get a very convenient description of the martingale $M$ as
$$M_t=\int_0^t\int_0^\infty
X^2_{s-}\,\1_{[0,I_{s-}]}(r)\;\CM(ds,dr).$$
Note that $\partial_tM_t=-\theta_t$ for $t<\tau$ and that $M_t=X^2_0$ for all $t\geq\tau$ since $I_t=0$ for all $t\geq\tau$.

Having understood the evolution of the one-colony model, we are prepared to study the model of finitely many colonies.

Let $\CA^{*,+}(k,l)=\CA^*(k,l)\1_{\{k\neq l\}}$. At each site $k$, the rate of immigration is $\theta_t(k)=\CA^{*,+}X^1_t(k)$ while the constant for the rate of emigration is $c(k)=-\CA^*(k,k)$.

As long as $X^2_t(k)>0$ (and hence $X^1_t(k)=0$), we have $\theta_t(k) = \CA^*X^1_t(k)$ and $\partial_t X^2_t(k)=\CA^*X^1_t(k)$; that is
\begin{equation}
\label{E1.10}
X^2_t(k)=X^2_0(k)+\int_0^t\CA^*X^1_s(k)\,ds\mfa k\mbs{and $t$ such that}X^2_t(k)>0.
\end{equation}
 On the other hand, if $X^2_{t_0}(k)=0$, then   $X^1(k)$ solves the equation
\begin{equation}
\label{E1.11}
 X^1_t(k)=X^1_{t_0}(k)+\int_{t_0}^t\CA^* X^1_s(k)\,ds\mfa t\geq t_0.
\end{equation}

We will now describe the simple situation where $S=\{-N,-N+1,\ldots,N-1,N\}$ and where $\CA$ is the $q$-matrix of a Markov chain on $S$ with only nearest neighbour jumps. Furthermore, we assume, $X^2_0(k)>0$ for $k>0$ and $X^2_0(k)=0$ for $k\leq0$.
Recall from the motivation that $X^1$ stands for wake frogs that move and $X^2$ for dormant frogs that stand still. Appealing to \eqref{E1.05} we could define independent random variables $W^k$, $k\in \{1,\ldots,N\}$, by
$$\P[W^k>x]=\frac{X^2_0(k)}{X^2_0(k)+x}\mfa x>0$$
and then define the process $(X^1,X^2)$ in a deterministic way using the $(W^k)$. However, this construction is a bit technical and does not differ too much from a similar construction for the model on the continuous site space $\R$ that we will present later. Hence, here we focus on the martingale measure $\CM$ introduced in \eqref{E1.09}.

Now for every $k\in S$, there is a martingale $M(k)$ but at a given time $t$ only one of them changes its values, the one at $k=\elle{t-}$, where
$$\elle{t}:=\min\big\{k:\,X^2_t(k)>0\big\}\wedge N.$$
Note that due to the nearest neighbour jumps, $\elle{t}$ is the unique point $k$ such that $\CA^*X^1_t(k)>0$ and $X^2_t(k)>0$ (unless $\langle X^2_t,\1\rangle=0$). In other word, the wake frogs invade the dormant sites one by one. $\elle{t}$ is the site that wake frogs try to invade at time $t$. Eventually, either all frogs are awake or the invasion gets stuck at some final point $u^*:=\sup_t\elle{t}<\infty$. In this case, we have $X^2_t(u^*)>0$ for all $t\geq0$.  Now define
$$I_t:=\frac{\CA^*X^1_t(\elle{t})}{X^2_t(\elle{t})}\,\1_{\{X^2_t(\elle{t})>0\}}.$$
Finally, let
\begin{equation}
\label{E1.12}
M_t(k):=\int_0^t\int_0^\infty \1_{\{\elle{s-}\}}(k)\,X^2_{s-}(k)\,\1_{[0,I_{s-}]}(r)\,\CM(ds,dr).
\end{equation}
It is not hard to check that these $M(k)$ are in fact orthogonal martingales and that the process $(X^1,X^2)$ defined by \eqref{E1.02} in terms of these martingales solves the martingale problem MP$_S$ posed in Theorem~\ref{T1}.

We refrain from giving a formal proof of this statement since we later study a similar statement for the continuous space process in more detail.

\subsection{Continuous space model with discrete colonies of dormant frogs}
\label{S1.2}

We will now define a model similar to the one presented in the previous section but with the site space $S=\R$. Initially, $X^2_0$ is a purely atomic finite measure with nowhere dense atoms which is supported by $(0,\infty)$, that is
\begin{equation}
\label{E1.13}
X^2_0=\sum_{i\geq 1}x_{i}\delta_{z_i}
\end{equation}
with $0<z_1<z_2<\ldots$ and $\sum_ix_i<\infty$.

We assume that $X^1_0$ has a density and is supported by $(-\infty,0]$. The mass transport of $X^1$ follows the heat flow, that is, $\CA=\CA^*=\roix$. With a slight abuse of notation we will denote the density of $X^1_t(dx)$ by $X^1_t(x)$.

Let $M_F$ be the space of finite measures on $\R$ equipped with the weak topology.
For $\mu\in M_F$ and $f$ a bounded measurable function on $\R$, denote
$$ \langle \mu,f\rangle = \mu(f)\equiv \int_{\R} f(x)\,\mu(dx).$$
We denote by $C_b(\R)$ and $C_c(\R)$ the spaces of bounded continuous functions, respectively compactly supported continuous functions. By $C_b^2(\R)\subset C_b(\R)$ and $C_c^2(\R)=C_b^2(\R)\cap C_c(\R)$ we denote the subspaces of twice continuously differentiable functions with bounded first and second derivative. By adding a superindex $+$, we further restrict the classes to nonnegative functions. For any metric space $U$, we denote by $D_U$ the space of c\`{a}dl\`{a}g functions $[0,\infty)\to U$ equipped with the Skorohod topology.

We now give a more formal description of $(X^{1}_t,\, X^{2}_t)$, $t\geq 0$, as $M_F$-valued processes.

The model we consider here is quite similar to the discrete space model with $S=\{-N,\ldots,N\}$ that was studied in Section~\ref{S1.1}: Define
$$\olt :=\sup\big\{x:\,X^2_t((-\infty,x])=0\big\},$$
where clearly $\olt=\infty$ if $X^2_t(\1)=0.$
If $X^2_t(\1)>0$ then $\olt$ describes  the position of the leftmost atom of $X^2_t$ at time $t$, say at $z_i$. Then $X^1_t$ (the wake frogs) solves the heat equation on $(-\infty,\olt )$ with Dirichlet boundary condition at $\olt $. The wake frogs at $\olt $ fall asleep (that is the $X^1$ mass killed at $\olt $ transforms into mass of $X^2_t(\{\olt \})$) until all sleeping frogs at $\olt $ wake up. At this instant, $\olt $ jumps to the next atom at $z_{i+1}$ and $X^1$ continues to solve the heat equation on $(-\infty,\olt )=(-\infty,z_{i+1})$ with Dirichlet boundary condition at $\olt =z_{i+1}$ and so on. Let $\tau^1=0$ and
$$\tau^i:=\inf\big\{t\geq0:\,\olt=z_i\big\}=\inf\big\{t\geq0:\,X^2_t(\{z_{i-1}\})=0\big\}.$$

More formally, we have
\begin{equation}\label{E1.14} \left\{\begin{aligned}
\partial_t X^1_t&= \frac{1}{2}\sdv X^1_t\quad\mbs{on} (-\infty,z_n),\mf \tau^{n}<t< \tau^{n+1},\\
X^1_t(x)&=0\mfa x\geq \olt ,
\end{aligned}
\right.
\end{equation}
It is well known that the above Dirichlet problem can be equivalently formulated as a solution of the heat equation with killing at
$z_n$. For $z\in\R$, let $S^{z}$ denote the semigroup of
heat flow with killing at $z$. That is, $(S^z_t)_{t\geq0}$ is the sub-Markov semigroup with density $p^z_t(x,y)=0$ if $z\in[x,y]$ or $z\in[y,x]$ and $p^z_t(x,y)=p_t(y-x)-p_t(2z-y-x)$ otherwise, by the reflection principle. Here $p_t(x)=(2\pi t)^{-1/2}\exp(-x^2/(2t))$ is the density of the standard heat kernel. For any measure $\mu$ on $\R$, we also define
$S^{z}_t\mu(dx)/dx= \int_{\R}p^z_t(y,x)\,\mu(dy), \;x\in\R$. With a slight abuse of notation we also write $S^z_t\mu(x)$ for the density $S^z_t\mu(dx)/dx$.
 Then
\begin{equation}
\label{E1.15}
X^1_t(x) = S^{z_n}_{t-\tau^n}X^1_{\tau^{n}}(x)\mfa \tau^{n}<t<\tau^{n+1},\; x\in\R.
\end{equation}
By applying the integration by parts formula to $X^1$ solving~\eqref{E1.14}, it is easy to derive that for any smooth function $\phi_t(x): \R_+\times\R \mapsto\R$ with compact support in the $x\in\R$ variable,
we have
\begin{equation*}
 -\int_{\tau^{n}}^t  \big\langle X^1_{s},\partial_s \phi_{s}
+\frac{1}{2}\partial^2_x \phi_{s}\big\rangle\,ds= \big\langle X^1_{\tau^{n}},\phi_{\tau^{n}}\big\rangle - \big\langle X^1_t,\phi_t\big\rangle+
 \int_{\tau^{n}}^t\left(\frac{1}{2}\partial^-_x X^1_s(z_n)\right) \phi_{s}(z_n)\,ds,\quad t\in (\tau^{n}, \tau^{n+1}),
\end{equation*}
where $\partial^-_x$ denotes  the left sided partial derivative, and hence,
\begin{equation}
\label{E1.16}
\big\langle X^1_t,\phi_t\big\rangle= \big\langle X^1_{\tau^{n}},\phi_{\tau^{n}}\big\rangle + \int_{\tau^{n}}^t  \big\langle X^1_{s},\partial_s \phi_{s}
+\frac{1}{2}\partial^2_x \phi_{s}\big\rangle\,ds +
 \int_{\tau^{n}}^t\left(\frac{1}{2}\partial^-_x X^1_s(z_n)\right) \phi_{s}(z_n)\,ds,\quad t\in (\tau^{n}, \tau^{n+1}).
\end{equation}
The above equation implies that formally
 $X^1$ solves the following equation on $t\in (\tau^{n}, \tau^{n+1})$
\begin{equation}
\label{E1.17}
\begin{aligned}
\partial_t X^1_t(x)&= \frac{1}{2}\sdv X^1_t(x)  +  \frac{1}{2}\partial^-_x X^1_s(z_n)\delta_{z_n}(x)\,,      \quad  x\in \R.
\end{aligned}
\end{equation}
Note that $\partial^-_x X^1_s(z_n)\leq 0$, and hence the last term in~\eqref{E1.16} describes the loss of mass of $X^1$ due to "killing" at $z_n$.
In what follows, for a function $f$,  $\partial^2_x f$ will denote the generalized second derivative of $f$ at $x$. If $\partial^2_x f$  is an absolutely continuous signed measure, then,
with a slight abuse of notation, we will write $\partial^2_xf(x)$ for the density of this measure; if  $\partial^2_x f$  has an atom at $x$, then $\partial^2_x f(\{x\})$ will denote the mass (possibly negative) of that atom at $x$. Note that for $f$  having one sided derivatives at all points and such that $f(y)=0$ for all $y\geq x$, $\partial^2_x f(\{x\})$ equals minus the left sided partial derivative $-\partial^-_xf(x)$ at $x$:
$$\partial^2_x f(\{x\})= -\partial^-_xf(x)\qquad\mbs{if}f(y)=0\mbs{for all}y\geq x.$$ With this notation at hand we can rewrite~\eqref{E1.17} as follows:
\begin{equation}
\label{E1.18}
\begin{aligned}
\partial_t X^1_t(x)&= \frac{1}{2}\sdv X^1_t(x)  -  \roix X^1_s(\{z_n\})\delta_{z_n}(x)\,,      \quad  x\in \R, \;\tau^{n}<t< \tau^{n+1}.
\end{aligned}
\end{equation}

From the above, we can easily derive an expression for the amount of mass of $X^1$ ``killed'' at $z_n$ by time~$t$.
The amount of mass of $X^1$ ``killed'' at $z_n$ by time $t<\tau^{n}$ is zero. By time $t\in(\tau^{n},\tau^{n+1})$ it equals the change of the total mass of $X^1$ in the time interval $[\tau^n,t]$, that is, by~\eqref{E1.18}, it is given by
\begin{equation}
\label{E1.19}
\big\langle X^1_{\tau^{n}}-X^1_t,\1\big\rangle=\int_{\tau^{n}}^t\roix X^1_s(\{z_n\})\,ds=
\int_{\tau^{n}}^t\roix(S^{z_n}_{s-\tau^n}X^1_{\tau^n})(\{x\})\rest{x=z_n}\,ds,
\end{equation}
where the second equality follows by~\eqref{E1.15}.

In our model, before time $\tau^{n+1}$ all the wake frogs that arrive at $z_n$ fall asleep; that is, all the ``green'' mass of $X^1$ that is killed at $z_n$ transforms into ``red'' mass at the same site. More formally,
$$X^2_t\big((-\infty,z^{n}]\big)=0\mfa t\geq \tau^{n+1},$$
$$X^2_t(A)=X^2_0(A)\mfa t\leq \tau^{n+1}, A\subset(z_{n},\infty),$$
and
\begin{equation}
\label{E1.20}
X^2_t(\{z_n\})  = X^2_0(\{z_n\}) \;+ \int_{\tau^{n}}^t\roix X^1_s(\{z_n\})\,ds\mf \tau^{n}<t<\tau^{n+1}\,.
\end{equation}
At time $\tau^{n+1}$, all the sleeping frogs at $z_n$ wake up, that is, the red mass at $z_n$ transforms at once into a green atom at $z_n$:
\begin{equation} \label{E1.21}
X^1_{\tau^{n+1}}=X^1_{\tau^{n+1}-}+X^2_{\tau^{n+1}-}(\{z_n\})\delta_{z_n}.
\end{equation}

Altogether the process looks like propagation of the wake frogs to the right with consecutive ``struggles'' with piles of dormant frogs. During the fights, the propagation of the wake frogs to the right stops and the amount of wake frogs decreases (as they fall asleep). However, if and when the wake frogs manage to wake up a pile of dormant frogs, they reactivate the formerly wake frogs that have fallen asleep at this spot. In addition the initially dormant frogs at this spot also wake up. Then the fight place moves to the next pile of sleeping frogs. As mentioned before, it is possible
that the wake frogs fail to activate the dormant frogs at some spot. In this case the remaining sleeping frogs will stay asleep forever. Otherwise, after waking up all sleeping frogs, the frogs propagate as a solution to the heat equation.

As long as the sleeping frogs are distributed according to a purely atomic measure $X^2_0(dx)$ with nowhere dense locations of the  atoms, the above construction is pretty simple. However, an immediate question that arises
is whether it is possible to construct such a process with  $X^2_0(dx)$ being absolutely continuous
measure? Answering this question is the main goal of the paper.

\subsection{Main Results}
\label{S1.3}
Our main objective is the construction and characterization  of the process $(X^1, X^2)$ with absolutely continuous $X^2_0(dx)$.
In the sequel we will use the following assumptions on the initial conditions of the process
$(X^1, X^2)$.

\begin{assumption}
\label{A1.6}
\begin{itemize}
\item[(i)] $X^1_0$ is an absolutely continuous finite measure with support in $(-\infty,0]$ and with bounded density $X^1_0(x)$.
\item[(ii)] $X^2_0$ is an absolutely continuous finite measure with support in $[0,1]$ and with density $X^2_0(x)$.
\item[(iii)] $X^2_0(x)>0$ for all $x\in(0,1)$ and $X^2_0(x)$ is continuous in $x\in(0,1)$.
\end{itemize}
\end{assumption}

The idea behind the construction of the process $(X^1, X^2)$ with these initial conditions is pretty simple. Since we do understand the behavior of  the process when $X^2_0$ is purely atomic, we will approximate the absolutely continuous initial
conditions $X^2_0$ by the purely atomic ones and see if and where the family of processes converges.

Let $\eta>0$. We define the family of approximating processes $(X^{1,\eta}, X^{2,\eta})$ as follows.
For any $\eta>0$, define
\begin{equation}
\label{E1.22}
\begin{aligned}
X^{1,\eta}_0&=X^1_0,\\
x^{2,\eta}_i&=X^2_0\big(\big((i-1)\eta, i\eta\big]\big)\mf i\geq 1,\\
X^{2,\eta}_0&=\sum_{i\geq 1} x^{2,\eta}_i\,\delta_{i\eta}.
\end{aligned}\end{equation}

Now let $(X^{1,\eta}_t, X^{2,\eta}_t)$ be the process defined in Section~\ref{S1.2}. The wake frogs represented by
$X^{1,\eta}$ try to activate the colonies of sleeping frogs represented by $X^{2,\eta}$.

To present our main result, recall the point process $\CN$ and the martingale measure $\CM$ from \eqref{E1.08} and \eqref{E1.09}. Note that
$$
\CM\big([0,t]\times A\big)_{t\geq0}
$$
is a martingale for any measurable $A\subset\R_+$ with finite Lebesgue measure.

For an $M_F\times M_F$-valued process $\bm{\mu}_t=(\mu^1_t\,,\mu^2_t), t\geq 0$, let $\elll{}{t}$  be the leftmost point of the support of $\mu^2_t$ if $\mu^2_t(\1)>0$ and the rightmost point of the support of $\mu^2_0$ otherwise. Then if $\partial^2_x
\mu^1_t(\elll{}{t})$ is well-defined,
 we can define
\begin{equation}
\label{E1.23}
 I(\bm{\mu}_t)=
\frac{\roix \mu^1_{t}(\{\elll{}{t}\})}{\mu^2_t(\{\elll{}{t}\})}\,\1_{\{\mu^2_t(\{\elll{}{t}\})>0\}},
\end{equation}
otherwise we set $ I(\bm{\mu}_t)=0$.

Our main result is as follows.
\begin{theorem}
\label{T3}
$\big((X^{1,\eta},X^{2,\eta})\big)_{\eta>0}$ is tight in $D_{M_F\times M_F}$ and any limit point for $\eta\downarrow0^+$ is a weak solution to the
following system of stochastic partial differential equations: for any $\phi_1,\phi_2\in  C_b^2(\R)$,
\begin{eqnarray}
\label{E1.24}
\left\{
\begin{array}{rcl}
X^{1}_{t}(\phi_1)&=&\displaystyle X^{1}_{0}(\phi_1)+\int_0^t
X^{1}_{s}\left(\frac{1}{2}\phi_1''\right)\,ds
 + M_t(\phi_1)
 \\[4mm]
 X^{2}_{t}(\phi_2)&=&\displaystyle X^{2}_{0}(\phi_2)-M_t(\phi_2),
\end{array}
\right.
\end{eqnarray}
where $M_t(\phi_i)$, $i=1,2$, are martingales derived from the orthogonal martingale measure $\CM$ by
\begin{equation}
\label{E1.25}
M_t(\phi):=\int_0^t\int_0^\infty Y_{s-}\, \phi(\elle{s-})\1_{[0, I(X_{s-})]} (a)
\cM(ds,da).
\end{equation}
Here $\elle{s}:=\inf\{x:\, X^2_s((-\infty,x])>0\}\wedge1$ and
$Y_s= X^2_s(\{\elle{s}\})$.
\end{theorem}
Let us give a few comments regarding the above result. As we will see, not only the limiting process $(X^1,X^2)$ solves the set of equations \eqref{E1.24},
 but also any of the  approximating processes $(X^{1,\eta},X^{2,\eta})$. The  only difference is
in  the set of the initial conditions. Note that we do not prove uniqueness of the solution of \eqref{E1.24} and \eqref{E1.25}. In fact, due to the absence of a duality relation (which helped in similar models), we do not see a viable way to prove uniqueness here.

We can also check some properties of the limiting process $(X^1,X^2)$. One of the interesting observations we have deals with the properties of some point measure induced by the jumps of
$(X^1,X^2)$, where the jumps are indexed by their \emph{spatial} location.
Let  $L$ denote the point process on $\R\times\R_+$ that describes the jumps of the total mass $X^2_t(\1)$ of $X^2_t$ defined as follows:
\begin{equation}
\label{E1.26}
L(dx, dz)=\sum_s\1_{\{\Delta X^2_s(\1)\not= 0\}} \delta_{(\elle{s-}, -\Delta X^2_s(\1))}(dx,dz).
\end{equation}
We will show that essentially $L$ is a Poisson point process. However, it may happen that the total mass of $X^1$ is not sufficient to wake up all sleeping frogs. In this case, the proliferation of wake frogs gets stuck at some random point $u^*$ and $L$ exhibits the Poisson points only up to this random position.

Let $\widetilde L$ be a Poisson point process on $\R\times(0,\infty)$ with intensity $X^2_0(dx)z^{-2}\,dz$.
 For $u\in\R$ and $s\geq0$, define
$$D_{u,s}:=\big\{(x,z):\, x\in[0,u], \, z>s+X^{2}_0([0,x))\big\}$$
and
\begin{equation*}
u^*:=\sup\big\{u:\, \widetilde L\big(D_{u,X^{1}_0(\1)}\big)=0\big\}.
\end{equation*}
Then we have the following result:
\begin{theorem}
\label{T4}
We can define $\widetilde L$ and $X$ on one probability space such that
$ L(A)=\widetilde L\big(A\cap([0,u^*)\times(0,\infty))\big)$ almost surely for all measurable $A\subset[0,\infty)\times(0,\infty)$ and we have
$$u^*=\sup\big\{\inf\big(\mathrm{supp}(X^2_t)\big):\,t\geq0\big\}=\inf\big\{u:\,X^2_t((-\infty,u])>0\mbs{for all}t\geq0\big\}.$$
\end{theorem}
From the definition of $D_{u,s}$ and the above theorem it is easy to see that a point of
 $\widetilde L$ in $D_{u,s}$ means that if initially there is a total mass of $s$ wake frogs, then the proliferation of wake frogs stops before the spatial point~$u$.
Also note that $u^*$ differs from $\sup\{\elle{t}:\,t\geq0\}$ only in that the latter is truncated at $1$.

One of the most interesting features of the limiting process is that it develops atoms of sleeping frogs at random points although it starts with an absolutely continuous distribution of sleeping frogs. In fact, all the random fluctuations in the model occur at the (moving) site with this unique atom.

Once the atom of sleeping frogs is created, the process behaves like the model described in Section~\ref{S1.2} (with only one colony of sleeping frogs) until the dormant frogs at this spot wake up. The main difficulty then lies in predicting where the next atom would appear.

\subsection{Outline}
\label{S1.4}
The rest of the paper is organized as follows. In Section~\ref{S2}, we start with some elementary considerations on the martingale problem and prove Theorems~\ref{T1} and \ref{T2}. In Section~\ref{S3}, we formulate the SPDE that the approximating process $(X^{1,\eta},X^{2,\eta})$ solves and we give a description of the process that is indexed by space rather than time, providing the proof of Theorem~\ref{T4}. In Section~\ref{S4}, we show tightness of the approximating process. Finally, in Section~\ref{S5}, we finish the proof of Theorem~\ref{T3} by showing convergence of the semimartingale characteristics.

\section{Discrete space martingale problems. Proofs of Theorems \ref{T1}, \ref{T2}}
\setcounter{equation}{0}
\setcounter{theorem}{0}
\label{S2}

\setcounter{equation}{0}
\setcounter{theorem}{0}
Let $Z=(Z_t)_{t\geq0}$ be an integrable nonnegative Markov process with respect to some filtration $\FF=(\CF_t)_{t\geq0}$. Assume $Z$ is of the form
$$Z_t=\cases{z_0+t,&\mfalls t<\tau,\\0&\msonst}$$
for some $\FF$-stopping time $\tau<\infty$ and some deterministic $z_0>0$.
\begin{lemma}
\label{L2.1}
$Z$ is a martingale if and only if $\P[\tau>t]=z_0/(t+z_0)$ for all $t>0$.
\end{lemma}
\begin{proofn}
If $\tau$ has the desired distribution, then for $t>s\geq0$, we have
$$\begin{aligned}
\E[Z_t\Given \CF_s]&=\frac{z_0+t}{z_0+s}\,Z_s\,\P[\tau>t\Given \CF_s]\\[2mm]
&=\frac{z_0+t}{z_0+s}\,Z_s\,\frac{\P[\tau>t]}{\P[\tau>s]}\\[1mm]
&=Z_s.
\end{aligned}
$$
On the other hand, if $Z$ is a martingale, then
\[\P[\tau>t]=\P[Z_t=t+z_0]=\frac1{t+z_0}\E[Z_t]=\frac1{t+z_0}\E[Z_0].\eope\]
\end{proofn}
For $u>0$ fixed, let $W(u)$ be a random variable with
\begin{equation}
\label{E2.01}
\P[W(u)>r]=\frac{u}{r+u}\mfa r>0.
\end{equation}
 Define the process $(U_t(u))_{t\geq0}$, by
$$U_t(u)=\cases{u+t,&\mfalls t<W(u),\\[1mm]0&\msonst.}
$$
Hence, $(U_t(u))_{t\geq0}$ is a continuous time Markov process on $[0,\infty)$ with generator
$$\CG f(x)=\left[f'(x)+\frac1x\big(f(0)-f(x)\big)\right]\1_{(0,\infty)}(x).
$$
That is, $U(u)$ grows linearly with slope $1$ until it collapses to $0$. By construction, $U(u)$ is a martingale.

We will need this process $U$ in the following proposition that is a preparation for proving Theorem~\ref{T2}. Recall that $E=[0,\infty)^2\setminus(0,\infty)^2$.
\begin{proposition}[MP$_1$]
\label{P2.2}
Consider the following martingale problem for the process $(X^1_t,X^2_t)_{t\geq0}$ on $E$ with initial condition $(X^1_0,X^2_0)$:
$$X^1_t=M_t+t,\qquad t\geq0,$$
for some zero mean martingale $M$ and
$$X^2_t=X^2_0-M_t,\qquad t\geq0.$$
Then $(X^1,X^2)$ is uniquely defined (in law) and
\begin{equation}
\label{E2.02}
\begin{aligned}
X^1_t&=t-U_t(X^2_0)+X^2_0\\
X^2_t&=U_t(X^2_0).
\end{aligned}
\end{equation}
\end{proposition}
\begin{proof}
Clearly, \eqref{E2.02} defines a solution of the martingale problem.

In order to show uniqueness, assume that $(X^1,X^2)$ is a solution of the martingale problem. Let
$$\tau:=\inf\big\{t\geq0:\,X^2_t=0\big\}.$$
Clearly, $X^2_t=0$ for $t\geq\tau$ since $X^2$ is a nonnegative martingale. Furthermore, $X^1_t=0$ for $t<\tau$. Note that
$$X^1_t+X^2_t=X^2_0+t.$$
Hence, $X^2$ is a martingale with
$$X^2_t=\cases{X^2_0+t,&\mfalls t<\tau,\\[2mm]0,&\msonst.}$$
By Lemma~\ref{L2.1}, the only solution is $X^2=U$.
\end{proof}
Now we come to the following generalization of the martingale problem (MP$_1$) where the input rate varies in time and there is also output proportional to $X^1_t$.
\begin{proposition}[MP$_2$]
\label{P2.3}
Let $\Theta$ be a locally finite measure on $[0,\infty)$ with $\Theta(\{0\})=0$. Let $c\geq0$. Let $(X^1_t,X^2_t)_{t\geq0}$ be a stochastic process on $E$ with initial condition $(X^1_0,X^2_0)$. Assume that
\begin{equation}
\label{5_12_1}
X^1_t=\Theta((0,t])-\int_0^tcX^1_s\,ds+M_t
\end{equation}
for some zero mean martingale $M$ and
\begin{equation}
\label{5_12_3}
X^2_t=X^2_0-M_t.
\end{equation}

Then $(X^1,X^2)$ is uniquely defined (in law) and
\begin{equation}
\label{E2.03}
\begin{aligned}
X^1_t&=\cases{\big(\Theta((0,\tau])+X^2_0\big)e^{-c(t-\tau)}+\int_{\tau}^te^{-c(t-s)}\Theta(ds),&\mfalls t\geq\tau,\\[2mm]
0,&\msonst,}\\[2mm]
X^2_t&=U_{\Theta((0,t])}(X^2_0),
\end{aligned}
\end{equation}
where
$$\tau=\inf\big\{t\geq0:\,X^2_t=0\big\}\;=\;\inf\big\{t\geq0:\,\Theta((0,t])\geq W(X^2_0)\big\}.$$
\end{proposition}
\begin{proof}
\textbf{Existence.} We show that \equ{E2.03} is a solution.
As a deterministic (or independent) time transform of a martingale, $X^2$ is a martingale. Hence,
$(M_t):=(X^2_0-X^2_t)$ is a martingale such that
$$M_t=-\Theta((0,t])\mbs{for}t<\tau,$$
 and
\begin{equation}
\label{5_12_2}
M_t = X_0^2\mf t\geq \tau.
\end{equation}
Therefore~$X^2$ satisfies~\eqref{5_12_3}.
This also gives that $X^1_t=0$ solves~\eqref{5_12_1} for $t<\tau$.
Recalling~\eqref{5_12_2} it is easy to see  from the theory of differential equations that $X^1_t=\big(\Theta((0,\tau])+X^2_0\big)e^{-c(t-\tau)}+\int_{\tau}^te^{-c(t-s)}\Theta(ds)$ solves~\eqref{5_12_1} for $t\geq\tau$.
From all this we infer that $(X^1,X^2)$ defined by~\eqref{E2.03}  is indeed a solution of (MP$_2$).

\textbf{Uniqueness.} As in the proof of (MP$_1$), $M_t$ must be constant for $t\geq \tau$ and hence must equal
$$M_t=M_\tau=X^2_0-X^2_\tau=X^2_0\mf t\geq \tau.$$
On the other hand, we have
$$M_t=-\Theta((0,t])\mbs{for}t<\tau.$$
This defines $M$ uniquely (in law).
\end{proof}
\paragraph{Proof of Theorem~\ref{T2}.}
Clearly, the proof of Theorem~\ref{T2} follows immediately from the above proposition with
$\Theta((0,t])=\int_0^t\theta_s\,ds$.

\hfill \qed

\begin{remark}
\label{R2.4}
Note that Proposition~\ref{P2.3} not only gives the proof of Theorem~\ref{T2}, but also gives the exact form of the solution.
\end{remark}

Now we are ready to give the
\paragraph{Proof of Theorem~\ref{T1}.}
In fact, having all the above results, the proof of Theorem~\ref{T1} is simple. Let $S_0:=\{k\in S:\,X^{2}_0(k)>0\}$. Denote by $\hat X^1$ the deterministic solution of
$$\hat X^1_0=X^1_0\mbs{and}\partial_t\hat X^1_t(k)=\1_{S\setminus S_0}(k)\sum_{l\in S\setminus S_0}\CA^*(k,l)\hat X^1_t(l).$$
That is, $\hat X^1$ follows the deterministic flow induced by $\CA^*$ but with killing at $S_0$.

Let $(\tilde X^1(k),\tilde X^2(k))$, $k\in S$, be independent solutions of MP$_0$ with $\tilde X^i_0(k)=X^i_0(k)$ for all $i=1,2$ and with (recall that $\CA^{*,+}(k,l)=\1_{\{k\neq l\}}\CA^*(k,l)$)
$$\theta_t(k):=\CA^{*,+}\hat X^1_t(k)$$
and
$$c(k):=-\CA^*(k,k).$$
Let $\tau:=\inf\big\{t\geq0:\, \tilde X^2_t(k)=0\mbs{for some}k\in S_0\big\}$. Note that $\theta_t(k)=\CA^{*,+}\tilde X^1_t(k)$ for $t<\tau$.   Let $k^*\in S_0$ be the unique element such that $\tilde X^2_\tau(k^*)=0$. Then we define $(X^1_t,X^2_t)=(\tilde X^1_t,\tilde X^2_t)$ for all $t<\tau$ and
 $$X^i_\tau(k)=X^i_{\tau-}(k)\mfa i=1,2,\;k\neq k^*$$
 and
 $$X^1_\tau(k^*)=X^2_{\tau-}(k^*),\qquad X^2_\tau(k^*)=0.$$
Now, use $(X^1_\tau,X^2_\tau)$ as the new initial state and proceed inductively as above to get a solution of MP$_S$.

\textbf{Uniqueness.} In order to show uniqueness of the solution of MP$_S$, note that $X^2_t(k)$ is a nonnegative martingale for any $k\in S$. Hence, if for some $k\in S$ and some $t_0$, we have $X^2_{t_0}(k)=0$, then $X^2_t(k)=0$ for all $t\geq t_0$. Hence, it is enough to show that in the above construction of a solution of MP$_S$, the stopping time $\tau$ and the position $k^*$ are unique in law. Hence, we assume that we are given a solution $(X^1,X^2)$ of MP$_S$. We define $S_0:=\{k\in S:\,X^2_0(k)>0\}$ and
$$\tau:=\inf\big\{t\geq0:\, X^2_t(k)=0\mbs{for some}k\in S_0\big\}.$$
Since the martingales $X^2(k)$, $k\in S$, are orthogonal, there is a unique element $k^*\in S_0$ such that $X^2_\tau(k^*)=0$.
Since $(X^2_{\tau\wedge t}(k))_{t\geq0}$ are orthogonal martingales by the optional stopping theorem, the hazard rate for a jump to $0$ for each of these martingales is
$$H_t(k):=\frac{\CA^*X^1_t(k)}{X^2_t(k)}\mfa t<\tau.$$
Since $(X^1_t,X^2_t)$ solves a deterministic set of differential equations with Lipschitz coefficients for $t<\tau$, we have uniqueness of
$(X^1_t,X^2_t)$ for $t<\tau$. Since the martingales are orthogonal, the hazard rate for $\tau$ is simply
$$H_t:=\sum_{k \in S_0}H_t(k).$$
Hence, $\P[\tau>t]=\exp\left(-\int_0^tH_s\,ds\right)$ and $\P[k^*=k\given \tau]=H_{\tau-}(k)/H_{\tau-}$. This shows uniqueness in law up to time $\tau$ and by iteration, we get uniqueness in law for all times.

This finishes the proof of Theorem~\ref{T1}.
\hfill\qed

\section{Characterization of the approximating process and proof of Theorem~\ref{T4}}
\setcounter{equation}{0}
\setcounter{theorem}{0}
\label{S3}
The approximating process $(X^{1,\eta},X^{2,\eta})$ was introduced in Section~\ref{S1.3} based on the construction in Section~\ref{S1.2}. Here we show that it satisfies a certain set of equations. Furthermore, in Section~\ref{S3.2}, we change the perspective and give a description in terms of the maximal amount of frogs that sleep at a given site before they wake up. This space indexed description results in a spatial point process description of the approximating process and yields the proof of Theorem~\ref{T4}.
\subsection{Martingale description of the approximating process}
\label{S3.1}
Let us define $\elleta{t}$  by
\begin{equation}
\label{E3.01}
\elleta{t}=\inf\big\{x:\; X^{2,\eta}_t((-\infty,x])>0\big\}\wedge \sup\mathrm{supp}\big(X^{2,\eta}_0\big).
\end{equation}
Note that $\elleta{t}$ is the left boundary of the support of $X^{2,\eta}_t$ as long as $X^{2,\eta}_t(\1)>0$. We use this specific definition in order to avoid that $\elleta{}$ jumps to $\infty$ at the instance where the last atom of $X^{2,\eta}$ vanishes.

From the description of the process in Section~\ref{S1.2}, we can decompose the process $X^{2,\eta}$ as follows:
\begin{equation}
\label{E3.02}
 X^{2,\eta}_t(dx)= X^{2,\eta}_t(\{\elleta{t}\})\delta_{\elleta{t}}(dx)+\1_{(\elleta{t},\infty)}(x)\,X^{2,\eta}_0(dx).
\end{equation}
For simplicity,  denote
\begin{equation}
\label{E3.03}
 Y^{\eta}_t\equiv X^{2,\eta}_t(\{\elleta{t}\}).
\end{equation}
Then we have
\begin{equation}
\label{E3.04}
 X^{2,\eta}_t(dx)= Y^{\eta}_t\delta_{\elleta{t}}(dx)+\1_{(\elleta{t},\infty)}(x)\,X^{2,\eta}_0(dx).
\end{equation}

In order to define the point process $\cN^{\eta}_{\Delta}$ of the sizes of the dormant colonies at the times when they wake up, we need the process of jumps
\begin{equation}
\label{E3.05}
(\Delta X^{2,\eta}_t)(\{\elleta{t-}\}):=\lim_{r\downarrow0}\left(X^{2,\eta}_{t}(\{\elleta{t-r}\})-X^{2,\eta}_{t-r}(\{\elleta{t-r}\})\right).
\end{equation}
We can now define point process $\cN^{\eta}_{\Delta}$ on $\R_+\times\R_+$ by
\begin{equation}
\label{E3.06}
 \cN^{\eta}_{\Delta}(ds, dz)=\sum_t\1_{\{(\Delta X^{2,\eta}_t)(\{\elleta{t-}\})\not= 0\}} \delta_{(t, -(\Delta
X^{2,\eta}_t)(\{\elleta{t-}\}))}(ds,dz).
\end{equation}

\begin{lemma}
\label{L3.01}
Let $\cN^{\eta,\prime}_{\Delta}$ be the compensator measure of $\cN^{\eta}_{\Delta}$. Then
\begin{equation}
\label{E3.07}
\cN^{\eta,\prime}_\Delta\big(ds,B\big)=
\1_B(Y^\eta_{s-})  \,I(X^\eta_{s-})\,ds\,\mf B\subset \R_+\mbox{ measurable.}
\end{equation}

\end{lemma}
\begin{proof}
By construction, $X^{2,\eta}_t(\{i\eta\})$ are orthogonal nonnegative martingales with hazard rates of a jump down to $0$ (i.e., of size $-X^{2,\eta}_t(\{i\eta\})$) given by $\roix X^{1,\eta}_t(\{i\eta\})/X^{2,\eta}_t(\{i\eta\})$ if $X^{2,\eta}_t(\{i\eta\})>0$. This hazard rate is positive only if $i\eta=\elleta{t}$ and in this case equals $I(X^\eta_{t})$. Hence, we get \eqref{E3.07}.
\end{proof}

Now we are going to derive the system of equations that describes $(X^{1,\eta}, X^{2,\eta})$.
Let $\cM^{\eta}_{\Delta}=\cN^{\eta}_{\Delta}
  -\cN^{\eta,\prime}_{\Delta}$ and for bounded measurable $\psi:\R_+\times\R\to\R$, define $M^{\eta}_t(\psi)$ by
  \begin{equation}
\label{E3.08}
M^{\eta}_t(\psi)=\int_{0}^{t}\int_{[0,\infty)} z\, \psi\big(s,\elleta{s-}\big)
\, \CM^{\eta}_{\Delta}\big(ds,dz\big).
\end{equation}
Clearly, $M^\eta_t(\psi)$ is a local martingale.
Note that
\begin{eqnarray}
\label{eq:18_12_1}
\int_{0}^{t}\int_{[0,\infty)} z\,
\, \CN^{\eta}_{\Delta}\big(ds,dz\big)
&&\leq \sum_{i=1}^{\lceil\eta^{-1}\rceil}\sum_{s}|\Delta X^{2,\eta}_s(\{i\eta\})|
\leq \lceil\eta^{-1}\rceil\,\sup_{s\geq0}\langle X^{2,\eta}_s,\1\rangle\\
\nonumber
&&\leq \lceil\eta^{-1}\rceil\langle X^1_0+X^2_0,\1\rangle).
\end{eqnarray}
At each point $\{i\eta\}$, we have
\begin{equation}
\label{E3.10}
\int_{0}^{t}\int_{[0,\infty)} \1_{\{\ell^\eta_{s-}=i\eta\}}\,z\,
\, \CN^{\eta,\prime}_{\Delta}\big(ds,dz\big)
\leq \sup_{s\geq0}X^{2,\eta}_s(\{i\eta\})\leq \langle X^1_0+X^2_0,\1\rangle.
\end{equation}
Hence, we infer
\begin{equation}
\label{E3.11}
|M^{\eta}_t(\psi)|\leq \|\psi\|_\infty\left(\int_{0}^{t}\int_{[0,\infty)} z\,
\, \CN^{\eta}_{\Delta}\big(ds,dz\big)
+
\int_{0}^{t}\int_{[0,\infty)} z\,
\, \CN^{\eta,\prime}_{\Delta}\big(ds,dz\big)\right)
\leq 2\lceil \eta^{-1}\rceil\,\|\psi\|_\infty\langle X^1_0+X^2_0,\1\rangle.
\end{equation}

As a bounded local martingale, $M^\eta_t(\psi)$ is in fact a martingale. We will use this martingale in the next lemma for functions $\phi$ instead of $\psi$ with no explicit time-dependence.

\begin{lemma}
\label{L3.02}
The process $(X^{1,\eta},X^{2,\eta})$ solves the following set of equations: For all $\phi_1,\phi_2 \in  C_b^2(\R)\,,$
\begin{eqnarray}
\label{E3.12}
X^{1,\eta}_{t}(\phi_1)&=& X^{1,\eta}_{0}(\phi_1)+\int_0^t X^{1,\eta}_{s}\left(\frac{1}{2}\phi''_1\right)\,ds
 + M^{\eta}_t(\phi_1)\\
\label{E3.13}
 X^{2,\eta}_{t}(\phi_2)&=& X^{2,\eta}_{0}(\phi_2)-
   M^{\eta}_t(\phi_2)
   \end{eqnarray}
where the martingales $M^\eta(\phi_i)$ are defined by \eqref{E3.08}.

\end{lemma}
\begin{proof}
This is an immediate consequence of the construction of the process $X^\eta$ and the definition of $\CM^\eta_\Delta$.
\end{proof}

\begin{lemma}
\label{L3.03}
There exists a unique $L^2$-martingale measure $M^{\eta}(ds,dx)$ on $\R_+\times\R$ such that
\begin{equation}
M^{\eta}_t(\phi)=\int_{0}^{t}\int_{\R} \phi(z)
\,M^{\eta}\big(ds,dz\big)\mfa \phi\in  C_b^2(\R).
\end{equation}
$M^{\eta}\big(ds,dx\big)$ is a pure jump martingale measure that has only a finite number of jumps (at most $\lceil \eta^{-1}\rceil $ jumps) and it fulfils
\begin{equation}
\label{E3.15}
\begin{aligned}
M^{\eta}_t(\psi)
&=\int_0^t\int_\R\psi(s,z)\,M^{\eta}(ds,dz)\\
&=\sum_{s\leq t}(-\Delta
X^{2,\eta}_s)(\{l^{\eta}_{s-}\}))\,\psi\big(s,l^{\eta}_{s-}\big) -  \int_{0}^{t}Y^{\eta}_{s-} I(X^\eta_{s-})\, \psi\big(s,l^{\eta}_{s-}\big)
\,ds\mf \psi\in C_b(\R_+\times\R).
\end{aligned}
\end{equation}
\end{lemma}
\begin{proof}
This is standard (note that in fact there is only a bounded number of bounded jumps, see the discussion in and after the proof of Lemma~\ref{L3.01}). See, e.g., Chapter 2 in \cite{bib:wal86} for more discussion of martingale measures.
\end{proof}

Alternatively, we can characterize the model via the process $(X^{1,\eta},Y^{\eta})$. Then we have a simple
corollary from the above lemma.
\begin{corollary}
\label{C3.04}
$(X^{1,\eta},Y^{\eta})$ solves the following equations. For any $\phi\in  C_b^2(\R)$,
\begin{eqnarray}
\label{E3.14}
 \langle  X^{1,\eta}_t\,,\phi\rangle&=& \langle  X^{1,\eta}_0\,,\phi\rangle+\int_{0}^{t} \Big\langle  X^{1,\eta}_s\,,
\frac{1}{2}\phi''\Big\rangle\,ds +
\int_{0}^{t}\int_{[0,\infty)} z \phi(\elleta{s-})
\, \CM^{\eta}_{\Delta}\big(ds,dz\big)\\
\label{E3.17}
 Y^{\eta}_t&=& X^{2,\eta}_0\big((-\infty,\elleta{t}]\big)-
\int_{0}^{t}\int_{[0,\infty)} z\,\CM^{\eta}_{\Delta}\big(ds,dz\big).
\end{eqnarray}
\end{corollary}

Note that the ``noise'' in the equations given above depends on the process itself. It is much more convenient to define the equations driven by a noise whose parameters are independent of the solutions.
To this end recall the point process $\cN$ and the corresponding martingale measure $\CM$ that were introduced in \eqref{E1.08} and \eqref{E1.09}.
\begin{lemma}
\label{L3.05}
We can define our process $X^\eta=(X^{1,\eta}, X^{2,\eta})$ and $\CM$ on one probability space such that the following SPDE holds:
For all $\phi_1,\phi_2 \in  C_b^2(\R),$ we have
\begin{eqnarray}
\label{E3.18}
X^{1,\eta}_{t}(\phi_1)&=& X^{1,\eta}_{0}(\phi_1)+\int_0^t X^{1,\eta}_{s}\left(\frac{1}{2}\phi_1''\right)\,ds
 + \int_0^t\int_{[0,\infty)} Y^{\eta}_{s-}\, \phi_1(\elleta{s-})\,\1_{[0, I(X^\eta_{s-})]} (a)
\cM(ds,da)
 \\
\label{E3.19}
 X^{2,\eta}_{t}(\phi_2)&=& X^{2,\eta}_{0}(\phi_2)-\int_0^t\int_{[0,\infty)} Y^{\eta}_{s-}\,\phi_2(\ell^\eta_{s-})\, \1_{[0, I(X^\eta_{s-})]} (a) \,\cM(ds,da).
\end{eqnarray}
\end{lemma}
\begin{proof}
By Lemma~\ref{L3.02}, $X^\eta$ solves \eqref{E3.12} and \eqref{E3.13}. Let $(t_n,  x_n)_{n\geq 1}$ be an arbitrary labeling of the points of the point process
$\CN_\Delta$. Let $\CN^1$ be a Poisson point process on $\R_+\times\R_+$, with intensity $dt\,dr$,
independent of $\CN_\Delta$ and $(X^{1,\eta},X^{2,\eta})$. Also let  $\{U_n\}_{n\geq 1}$ be a sequence of independent random variables
uniformly distributed on $(0,1)$ which are also independent of  $\CN_\Delta$ and $(X^{1,\eta},X^{2,\eta})$.

 Define the new point process $\CN$ on $\R_+\times\R_+$ by
\begin{equation}
\label{E3.20}
\begin{aligned}
\CN(ds,dr)
=&\,\sum_{n\geq 1}
  \delta_{\left(t_n,U_n I(X^\eta_{t_n-})\right)}(ds,dr)\\
&\,+\sum_{n\geq 1} \1_{\{r>  I(X^\eta_{t_n-})  \}} \CN^1 \big(ds,dr\big).
\end{aligned}
\end{equation}
Both summands in \equ{E3.20} are predictable transformations of point processes of class (QL) (in the sense of \cite[Definition II.3.1]{IkedaWatanabe1989}); that is, they possess continuous compensators. Standard arguments yield that they are hence also point processes of class (QL). A standard computation shows that the compensator measures are given by
$$  \,\1_{\{r\leq   I(X^\eta_{s-}) \}} \,ds\,dr
\mbu \,\1_{\{r>   I(X^\eta_{s-}) \}} \,ds\,dr,
$$
respectively. Hence, $\CN$ is a point process of class (QL) and has the deterministic and absolutely continuous compensator measure
$dr\,dt$.
By \cite[Theorem II.6.2]{IkedaWatanabe1989}, we get that $\CN\,$ is thus a Poisson point process with intensity $dt\,dr$.
\end{proof}

\subsection{Set indexed description  of the approximating process, proof of Theorem~\ref{T4}}
\label{S3.2}
The aim of this section is to prove Theorem~\ref{T4}. Note that in that theorem, the perspective has changed from a stochastic process indexed by time to a stochastic process indexed by space. In fact, we can consider the struggle between dormant (red) and wake (green) frogs at a given site as a continued gambler's ruin problem. As long as there is an amount of, say, $y$ dormant frogs at a given site, the arriving infinitesimal amount $dx$ of wake frogs have a chance $dx/y$ to activate all frogs at that site. Otherwise the arriving frogs fall asleep and the total amount of dormant frogs at the site increases to $y+dx$. Once all frogs at a given site have woken up, the wake frogs can proceed to the next pile of sleeping frogs (to the right).

Clearly, we can describe this process by determining for each spatial point \emph{in advance} the amount of wake frogs that is needed to wake up the different piles of dormant frogs. This point process is $\widetilde L$ for the limiting process and is $\widetilde L^\eta$ for the approximating processes. We claim that we can construct $\widetilde L^\eta$ as a simple and natural function of $\widetilde L$ and that $\widetilde L^\eta$ converges to $\raisebox{4pt}{\rule{0pt}{7pt}}\widetilde L$ almost surely.

Before we start with the formal statements, we make the following considerations. We consider the wake frogs as playing the gambler's ruin problem described above successively against a finite number of piles of sleeping frogs. These piles have the sizes $x_1,\ldots,x_n>0$. Assume that $W_i$ is the random amount of wake frogs it takes to wake up the $i$th pile, $i=1,\ldots,n$. Then clearly, $W_1,\ldots,W_n$ are independent and we have
\begin{equation}
\label{E3.21}
\P[W_i>r]=\frac{x_i}{x_i+r}\mf r>0.
\end{equation}
Now we ask: How many wake frogs are needed to wake up successively all dormant frogs? After waking up the first pile of dormant frogs, we have $W_1+x_1$ wake frogs to try and wake up the second pile of dormant frogs. If $W_2<W_1+x_1$ then the wake frogs can wake up all the frogs in pile 2 without additional help from other wake frogs and then we have $W_1+x_1+x_2$ frogs who try to wake up the frogs in pile number 3. If, however, $W_2\geq W_1+x_1$, then we need additional wake frogs $W_2-(W_1+x_1)$ before waking up pile 2. In this case, we have $W_2+x_2$ frogs to wake up pile number 3.
Summing up, we need $$\max\{W_1,W_2-x_1\}$$
wake frogs to wake up the frogs in the piles 1 and 2. Iterating this, we see that we need
\begin{equation}
\label{E3.22}
\max\big\{(W_i-(x_1+\ldots+x_{i-1}):\,i=1,\ldots,n\big\}
\end{equation}
initially wake frogs to wake up all dormant frogs.

Now we will show how to construct a set of random variables $\wtW^1,\ldots,\wtW^n$ from some Poisson point process such that
\begin{equation}
\label{E3.23}
(\wtW_1,\ldots,\wtW_n)\stackrel{d}{=}  (W_1,\ldots,W_n).
\end{equation}
Consider the Poisson point process $J$ on $[0,\infty)\times(0,\infty)$ with intensity measure $dz\otimes r^{-2}dr$. The points are thought of as the amounts of wake frogs needed to wake up infinitesimal dormant frogs that are situated at the spatial points $z$. For simplicity, we enumerate the points of $J$ in an arbitrary way as $(z_i,r_i)$. We write $\tilde x_i=x_1+\ldots+x_i$. For $s>0$, define the set
$$D_{i,s}=\big\{(z,r):\,z\in(\tilde x_{i-1},\tilde x_i],\,r>s+(z-\tilde x_{i-1})\big\}$$
Motivated by \eqref{E3.22}, we define
$$\wtW_i:=\sup_{(z_j,r_j):\, z_j\in(\tilde{x}_{i-1},\tilde x_i]}\big(r_j-(z_j-\tilde{x}_{i-1})\big)=\inf\big\{s:\,J(D_{i,s})=0\big\}.$$
\begin{lemma}
\label{L3.06}
The random variables $\wtW_1,\ldots,\wtW_n$ are independent and $\P[\wtW_i> r]=\frac{x_i}{r+x_i}$ for all $i=1,\ldots,n$ and $r>0$. That is, \eqref{E3.23} holds.
\end{lemma}
\begin{proofn}
The independence is obvious as the points are taken from disjoints intervals.
In order to compute the distribution of $\wtW_i$, for $s>0$, we compute
\begin{align}
\label{E3.24}
\P[\wtW_i\leq s]&= \P\left[J(D_{i,s})=0\right]\nn\\
&= \exp\left(-    \int_{\tilde x_{i-1}}^{\tilde x_i} \int_{s+  (z-\tilde x_{i-1})}^{\infty}  r^{-2} \,dr\, dz \right)
\\
&=
\exp\left[\log(s)-\log(s+  x_i)\right]  \nn\\
&=
\frac{s}{s+x_i}.
\eope\end{align}
\end{proofn}

The lemma shows that we can start from infinitesimal dormant frogs and lump them together to build piles of finite size. Similarly, we can go backwards and split finite piles into smaller and smaller pieces to obtain the process $J$. We will formulate this in the slightly more general situation where $J$ has the density $f(z)dz\otimes r^{-2}dr$ for some bounded and nonnegative function $f$. Let $\mu$ be the measure on $[0,\infty)$ with density $f$. Furthermore, let $\eta>0$ and define

$$\wtW^\eta_i:=\sup_{(z_j,r_j):\, z_j\in((i-1)\eta, i\eta]}\left(r_j-\mu([(i-1)\eta,z_j))\right).$$

\begin{proposition}
\label{P3.07}
The random variables $\wtW^\eta_i$, $i=1,2,\ldots$ are independent and $$\P[\wtW^\eta_i> r]=\frac{\mu([(i-1)\eta,i\eta))}{\mu([(i-1)\eta,i\eta))+r}\mfa r>0.$$
Furthermore, the point process
$$J^\eta:=\sum_{i}\delta_{(i\eta,\wtW^\eta_i)}$$
converges almost surely to $J$ (in the vague topology of Radon measures on $[0,\infty)\times(0,\infty)$).
\end{proposition}
\begin{proof}
The independence of the $\wtW^\eta_i$ and the specific form of their distribution is immediate from Lemma~\ref{L3.06}.
In order to show convergence of $J^\eta$, it is enough to show
\begin{equation}
\label{E3.25}
J^\eta(C)\to J(C)\mbs{as}\eta\to0\mfasts\end{equation}
for sets $C$ of the form $C=[x_1,x_2]\times[s,\infty)$ for some $x_2>x_1\geq0$ and $s>0$. Note that $J(C)$ is finite almost surely.
We first define a point process $\tilde J^\eta$ that is similar to $J^\eta$ but a little simpler. For $\eta>0$, define $$M^\eta_i:=\max\big\{r_j:\,z_j\in((i-1)\eta,i\eta]\big\}=\inf\big\{s'>0:\,J\big(((i-1)\eta,i\eta]\times[s',\infty)\big)=0\big\}$$
and let
$$\tilde J^\eta:=\sum_{i}\delta_{(i\eta,M^\eta_i)}.$$
Since the points of $J$ in the set $C$ are discrete, the points of $\tilde J^\eta$ in $C$ approximate the points of $J$ in $C$ for $\eta>0$ small enough. To make this precise, note that
since the intensity measure of $J$ has a density, almost surely there exists a (random) $\varepsilon>0$ such that
$$J\big([x_1-\varepsilon,x_1]\times[s,\infty)\big)=J\big([x_2-\varepsilon,x_2]\times[s,\infty)\big)
=J\big([x_1-\varepsilon,x_2]\times[s-\varepsilon,s)\big)=0$$
and
$$ J\big([x,x+\varepsilon]\times[s,\infty)\big)\leq 1\mfa x\in[x_1,x_2-\varepsilon].$$
For  $\eta\in(0,\varepsilon)$, we then have $\tilde J^\eta(C)= J(C)$.

It remains to compare $\tilde J^\eta$ and $J^\eta$. The points in $\tilde J^\eta$ have a slightly larger  second coordinate (at most $\varepsilon/2>0$ larger if $\eta$ is small enough) and have thus possibly more points in sets of the form $C$. These additional points must originate in points of $J$ in $[x_1-\varepsilon,x_2]\times(s-\varepsilon,s)$. By assumption, however, there are no such points, and hence $J^\eta(C)=\tilde J^\eta(C)$.
More formally, we have $M^\eta_i\geq \wtW^\eta_i\geq M^\eta_i-\eta\,\|f\|_\infty$. This shows that
$$J^\eta(C)\leq\tilde J^\eta(C)\leq J^\eta(C)+J\big([x_1-\varepsilon,x_2]\times(s-\varepsilon,s)\big)=J^\eta(C)$$ if $0<\eta<\varepsilon\wedge(\varepsilon/2\|f\|_\infty)$ and hence it shows \eqref{E3.25}.
\end{proof}

Now we come back to our process $(X^{1,\eta},X^{2,\eta})$. Here we will address the following question: What is the distribution of the amount of wake frogs that is  needed to wake up a given pile of sleeping frogs. Recall the definition of the point process $\cN^\eta_{\Delta}$ introduced in \eqref{E3.06}. Let $(t_n,x_n)$ be the points of that process and we label them in a way that
$$ t_1<t_2<\ldots.$$
Now define the new point process
\begin{equation}\begin{aligned}
L^\eta( dx,dr)&\equiv \sum_{i}\delta_{(\elleta{t_i-},|(\Delta X^{2,\eta}_{t_i})(\{\elleta{t_i-}\})| -
X^{2,\eta}_{0}(\{\elleta{t_i-}\}))}(dx,dr)\\
&= \sum_{i}\delta_{(i\eta,V^\eta_i)}(dx,dr),
\end{aligned}\end{equation}
where
$$V^{\eta}_i\equiv |\Delta X^{2,\eta}_{t_i}(\{i\eta\})| -
X^{2,\eta}_{0}(\{i\eta\})= |\Delta X^{2,\eta}_{t_i}(\{i\eta\})| - x^{2,\eta}_i$$
is exactly the amount of wake frogs that arrived at $i\eta$ before waking up the dormant frogs in the $i$th pile.
Let us characterize these random variables $V^\eta_i$ and the point process $L^\eta$.

Recall the definition of $W(x)$, $x>0$ from \eqref{E2.01}.
Now given the initial measure of dormant frogs $X^{2,\eta}_0=\sum_{i\geq 1} x^{2,\eta}_i\delta_{i\eta}$, we define the following sequence of independent random variables $W^\eta_i,\; i=1,2\ldots,$ such that
\begin{equation}
\label{E3.27}
W^\eta_i \stackrel{d}{=}  W(x^{2,\eta}_i)\mfa i\geq 1.
\end{equation}
Recall that $W^\eta_i$ is (in distribution) the amount of wake frogs needed to activate the dormant frogs at $i\eta$. However, since the wake frogs possibly do not suffice to wake up all dormant frogs, we see only some of the $W^\eta_i$ realized as values of $V^\eta_i$. Note that, for any
$i$,  we have
\begin{equation}\begin{aligned}
V^\eta_i&\leq
 \sum_{j<i} x^{2,\eta}_j +\langle X^{1,\eta}_0,\1\rangle\\
&= X^{2,\eta}_0((-\infty,\eta i))+\langle X^{1,\eta}_0,\1\rangle.
\end{aligned}\end{equation}
Recall that by Assumption~\ref{A1.6}(iii) and the stepsize $\eta$, our $L^\eta$ has a finite number of atoms, whereas the ``last'' atom is spatially located at $\elleta{\infty}$ which is either at
$\lceil \eta^{-1}\rceil \eta$ or at the location of the leftmost pile of dormant frogs that will never wake up because there are not enough wake frogs to activate them.
Define
\begin{equation}
i^{\eta,*}=\inf\Big\{i=1,\ldots, \lceil \eta^{-1}\rceil:\, W^\eta_i\geq X^{2,\eta}_0((-\infty,\eta i))+\langle X^{1,\eta}_0,\1\rangle\Big\}
\end{equation}
which might be infinite. This is the analogous quantity to $u^*$ from Theorem~\ref{T4}.
Now we define an auxiliary point process based on the above random variables:
\begin{equation}\begin{aligned}
\widetilde L^\eta  &\equiv \sum_{i=1}^{\lceil \eta^{-1}\rceil}\delta_{(i\eta,W^\eta_i)}.
\end{aligned}\end{equation}

We have the following lemma.
\begin{lemma}
\label{L3.08}
For all $\eta>0$,
\begin{equation}
\label{E3.31}
L^\eta (\ARG)\stackrel{d}{=}\widetilde L^\eta(\ARG\cap([0,i^{\eta,*}\eta)\times(0,\infty))).
\end{equation}
\end{lemma}
\begin{proof}
The proof of the lemma is simple and thus is omitted.
\end{proof}

Now we are ready to finish the

\paragraph{\bf Proof of Theorem~\ref{T4}.}
Recall $\widetilde L$ from Theorem~\ref{T4}. By Proposition~\ref{P3.07} with  $\mu=X^2_0$, $J \stackrel{d}{=}   \widetilde L$ and $J^\eta \stackrel{d}{=}   \widetilde L^\eta$,
we can assume that all the processes $\widetilde L$ and $\widetilde L^\eta$, $\eta>0$, are constructed on one probability space such that $\widetilde L^\eta\to \widetilde L$ almost surely. It is simple to see that $i^{\eta,*}\eta$ converges to $u^*$ almost surely. Hence, we can choose a subsequence $\eta_k\downarrow0$ such that $L^{\eta_k}$ converges almost surely to some $\widehat L$. By Theorem~\ref{T3} and after taking another subsequence of $(\eta_k)$ if needed, and by the properties of convergence in Skorohod space, we see that all jumps of $X^{2,k}$ of size at least $\varepsilon$ converge in size and position to the jumps of $X^2$ for all $\varepsilon>0$. In other words, we have $$\widehat L=L\mfasts.$$

This finishes the proof of Theorem~\ref{T4}.
\hfill\qed

We close this section with a proposition that will be used in Section~\ref{S4.2} for proving tightness of the
approximating processes.

\begin{proposition}
\label{P3.09}
Recall the sequence of random variables $W^\eta_i\,, i\geq 1,$ defined in~\eqref{E3.27}.
Let $a\in (0,10^{-1})$ be arbitrary. Let $\delta>0$ and let
$$\underline x_{2}=\underline x_{2,a,\delta}:=\delta^{-1}\inf\big\{X^2_0((x,x+\delta/2)):\,x\in(a,1-a)\big\}.$$
Then, for all $\eta>0$ sufficiently small, we have
\begin{equation}
\label{E3.32}
\P\left[\exists j\in \left\{ \left\lceil \frac{a}{\eta}\right\rceil,\ldots,\left\lceil \frac{1-a}{\eta}\right\rceil \right\}:\, \max _{i=1, \ldots, \lceil \delta/\eta\rceil }W^\eta_{i+j}<\delta^2\right]\leq \frac{2}{\delta}\,e^{-\underline x_2/(2\delta)}.
\end{equation}
For $\delta\leq \underline x_2/(12\log(12/\underline x_2))$, the right hand side of \eqref{E3.32} is bounded by $e^{-\underline x_2/(3\delta)}$.
\end{proposition}

For the proof of Proposition~\ref{P3.09} we need the following lemma.
\begin{lemma}
\label{L3.10}
Let $n\in\N$ and assume that $A_1,\ldots,A_n$ are independent events. Let $k\in\{2,\ldots,\lfloor n/2\rfloor\}$ and
$$c:=\min_{j=0,\ldots,n-\lceil k/2\rceil}\sum_{i=1}^{\lceil k/2\rceil}(1-\P[A_{j+i}]).$$
Then
\begin{equation}
\label{E3.33}
\P\left[\bigcup_{l=0,\ldots,n-k}\;\bigcap_{i=1}^{k}A_{l+i}\right]\leq \frac{2n}{k}\exp(-c).
\end{equation}
\end{lemma}
\begin{proof}
Each of the sets $\{l+i:\,i=1,\ldots,k\}$, $l=0,\ldots,n-k$, contains at least one of the sets $\{r\lceil k/2\rceil+1,\ldots,(r+1)\lceil k/2\rceil\}$, $r=0,\ldots,\lfloor n/\lceil k/2\rceil\rfloor-1$.
Hence, the left hand side in \eqref{E3.33} is bounded by
$$\begin{aligned}
\sum_{r=0}^{\lfloor n/\lceil k/2\rceil\rfloor-1}\P\left[\bigcap_{i=1}^{\lceil k/2\rceil}A_{r\lceil k/2\rceil +i}\right]
&\leq \left\lfloor n/\lceil k/2\rceil\right\rfloor\max_{j=0,\ldots,n-\lceil k/2\rceil}\prod_{i=1}^{\lceil k/2\rceil}\P[A_{j+i}]\\
&\leq \frac{2n}{k} \max_{j=0,\ldots,n-\lceil k/2\rceil}\prod_{i=1}^{\lceil k/2\rceil}\exp\big(-(1-\P[A_{j+i}])\big)\\
&= \frac{2n}{k} \exp(-c).
\end{aligned}$$
\end{proof}

Now we are ready to give

\paragraph{\bf Proof of Proposition~\ref{P3.09}.}

Recall from Assumption~\ref{A1.6} that the density $X^2_0(x)$ of $X^2_0$ is bounded on compact subsets of $(0,1)$. Hence, $c_a:=\sup_{x\in[a/2,1-a/2]}(X^2(x))<\infty$. Let $\eta\in(0,\delta^2/c_a)$. Now, for $i$ and $j$ from \equ{E3.32}, we have $x^{2,\eta}_{i+j}\leq \eta c_a\leq \delta^2$ and
$$\sum_{i=1}^{\lceil\delta/(2\eta)\rceil}
x^{2,\eta}_{i+j}\geq X^2_0((j\eta,j\eta+\delta/2))\geq \underline x_2\delta.$$
Hence, we have
$$
\sum_{i=1}^{\lceil\delta/(2\eta)\rceil}\big(1-\P\big[W^\eta_{i+j}<\delta^2\big]\big)
\;=\;\sum_{i=1}^{\lceil\delta/(2\eta)\rceil}
\frac{x^{2,\eta}_{i+j}}{\delta^2+x^{2,\eta}_{i+j}}
\;\geq\;\sum_{i=1}^{\lceil\delta/(2\eta)\rceil}
\frac{x^{2,\eta}_{i+j}}{2\delta^2}\;\geq\; \frac{\underline x_2}{2\delta}.
$$
The claim now follows from Lemma~\ref{L3.10} with $c\geq \frac{\underline x_2}{2\delta}$, $k=\lceil \delta/\eta\rceil$ and $n=\lceil (1-a)\eta^{-1}\rceil-\lceil a\eta^{-1}\rceil+1<\eta^{-1}$, hence $2n/k\leq 2/\delta$. (Note that $\lceil k/2\rceil = \lceil \delta/(2\eta)\rceil$.)
\hfill \qed

\section{Tightness of the approximating processes}
\label{S4}
In order to show tightness of the approximating processes, it is crucial to have a control on the motion of the interface $\elleta{t}$ between wake and sleeping frogs. In Section~\ref{S4.1}, we derive some useful bounds on the approximating processes and we show that for any limiting point $(X^1,X^2)$ of $(X^{1,\eta}, X^{2,\eta})$ neither $X^1$ has jumps down, nor
$X^2$ has jumps up.
 In Section~\ref{S4.2}, we use these bounds to control the motion of the interface and to finally infer tightness of the approximating processes.
\setcounter{equation}{0}
\setcounter{theorem}{0}
\subsection{Preliminary results}
\label{S4.1}
This section is devoted to showing for any limiting point
$(X^1,X^2)$  of
$(X^{1,\eta}, X^{2,\eta})$ (if it exists) that neither $X^1$ has jumps down nor $X^2$ has jumps up.
Recall our Assumption~\ref{A1.6}. From this assumption and our construction of
  the approximating process $(X^{1,\eta}, X^{2,\eta})$,  it is clear that the initial measure $X^{2,\eta}_0$ is concentrated on atoms sitting on the set $\{i\eta,\;i= 1,\ldots, \lceil\eta^{-1}\rceil\}$.

Denote
\begin{equation}
\label{E4.01}
\overline x_1\equiv  \sup_x X^1_0(x)\mbu\overline x_2\equiv  \sup_x X^2_0(x).
\end{equation}
\medskip

One of the main problems is to prove that in the limit neither $X^1$ has jumps down nor $X^2$ has jumps up.
To this end we need to find a bound on the mass of the process $X_1$ in any fixed small space interval.

\begin{proposition}
\label{P4.01}
For any $\deltaone \in(0,1)$, for all $0<\delta<\deltaone /\max\{1,\,240\,\overline x_2,\,40\,\overline x_1\}$ and all $0<\eta\leq\delta$, we have
\begin{equation}
\label{E4.02}
\P\left[\sup_{t\geq 0} X^{1,\eta}_t\big([x_0-\delta,x_0]\big)\geq \deltaone \right]
 \leq 5520\,\frac{\sqrt{\delta}}{\deltaone }\mfa x_0\in\R.
\end{equation}
\end{proposition}
\begin{proof}
First note that
$$\1_{[x_0-\delta,x_0]}(x)\leq \sqrt{2\pi e}\, \delta\, p_{\delta^2}(x_0-x)\leq5\, \delta\, p_{\delta^2}(x_0-x)\mfa x\in \R,\,\delta>0$$
where $p_t$ is the heat kernel and satisfies $\frac12\partial^2_xp_t(x)=\partial_tp_t(x)$.
Thus, it is clear that it is enough to get an appropriate bound on
$$ \P\left[\sup_{t\geq 0} X^{1,\eta}_t\big(\delta p_{\delta^2}(x_0-\ARG)\big)\geq \deltaone /5\right].$$

We will bound the appropriate probabilities for all the terms on the right hand side of \equ{E3.12}. By the choice of $\delta$, the first term is
\begin{equation}
\label{E4.03}
X^{1,\eta}_0(\delta\, p_{\delta^2}(x_0-\ARG))\leq \delta\,\overline x_1< \frac{\xi}{40}.
\end{equation}
The next term we bound is the martingale term.
\begin{equation}
\label{E4.04}
\begin{aligned}
 \P\left[\sup_{t\geq 0}|M^{\eta}_t( \delta p_{\delta^2}(x_0-\ARG))|\geq \deltaone /20\right]&=
  \P\left[\sup_{t\geq 0}\Big|X^{2,\eta}_t(\delta p_{\delta^2}(x_0-\ARG))-X^{2,\eta}_0(\delta p_{\delta^2}(x_0-\ARG))\Big|\geq \deltaone /20\right]\\
&\leq   \P\left[ \big|X^{2,\eta}_0(\delta p_{\delta^2}(x_0-\ARG))\big|\geq \deltaone /40\right]\\
&\quad+
\P\left[\sup_{t\geq 0}\big|X^{2,\eta}_t( \delta p_{\delta^2}(x_0-\ARG))\big|\geq \deltaone /40\right].
\end{aligned}\end{equation}
Denote by $\vartheta(x)=\sum_{n=-\infty}^\infty e^{-x\pi n^2}$, $x>0$, (a variation of) Jacobi's theta-function. By Ramanujan's formula (see \cite[p. 525]{WW66}), we have $\vartheta(1)=\frac{\pi^{1/4}}{\Gamma(3/4)}\approx1.086$. By Jacobi's equality (see, e.g., \cite[page 307, line 23]{Jacobi1828}), we have $\vartheta(x)=\vartheta(1/x)/\sqrt x$ and hence for $x\leq1$
\begin{equation}
\label{E4.05}
\vartheta(x)=\frac{1}{\sqrt{x}}\vartheta(1/x)\leq\frac{1}{\sqrt{x}}\vartheta(1)\leq \frac{1.086}{\sqrt{x}}.
\end{equation}
We use this in the third line with $x=\frac{\eta^2}{2\pi\delta^2}$ to get
\begin{equation}
\label{E4.06}
\begin{aligned}
X^{2,\eta}_0( p_{\delta^2}(x_0-\ARG))
&\leq \overline x_2\eta\,\frac{1}{\sqrt{2\pi\delta^2}}
\sum_{n=-\infty}^\infty\exp\left(-\frac{(x_0-n\eta)^2}{2\delta^2}\right)\\
&\leq \overline x_2\eta\,\frac{1}{\sqrt{2\pi\delta^2}}
\left(1+\vartheta\left(\frac{\eta^2}{2\pi\delta^2}\right)\right)\\
&\leq \overline x_2\eta\,\frac{1}{\sqrt{2\pi\delta^2}}
\left(1+\frac{2\delta}{\eta}\sqrt{2\pi}\,\vartheta(1)\right)\\
&= \overline x_2
\left(\frac\eta\delta\frac{1}{\sqrt{2\pi}}+\vartheta(1)\right)\leq 2\,\overline x_2.
\end{aligned}
\end{equation}
Here, the last inequality follows since $\eta\leq\delta$. Summing up, we have
\begin{equation}
\label{E4.07}
\|S_{\delta^2}X^{2,\eta}_0(\ARG)\|_\infty\leq 2\,\overline x_2.
\end{equation}
Here $S_t$ is the heat semigroup, that is, $S_t\phi(x) = \int_\R p_t(x-y)\phi(y)\,dy$ for any integrable function $\phi$ and \mbox{$S_t\mu(x) = \int_\R p_t(x-y)\,\mu(dy)$} if $\mu$ is a measure.
Since $\eta\leq\delta\leq\frac{\deltaone }{40\,\overline x_2}$,
the first term on the right hand side of \eqref{E4.04} equals zero.
As for the second term, $X^{2,\eta}_t(\delta\, p_{\delta^2}(x_0-\ARG))$ is a non-negative
martingale, hence, by Doob's inequality, we get
\begin{equation}
\begin{aligned}
\P\left[\sup_{t\geq 0}|X^{2,\eta}_t(\delta p_{\delta^2}(x_0-\ARG))|\geq \deltaone /40\right]
&\leq \frac{\E\big[X^{2,\eta}_0( \delta  p_{\delta^2}(x_0-\ARG))\big]}{\deltaone /40}\\
\nonumber
&=\frac{\E\big[\delta S_{\delta^2}X^{2,\eta}_0(x_0)\big]}{\deltaone /40}\\
&\leq \frac{80\,\overline x_2 \delta}{\deltaone },
\end{aligned}\end{equation}
where the last inequality follows from~\eqref{E4.07}.
These bounds imply  that
\begin{equation}
\label{E4.08}
   \P\left[\sup_{t\geq 0}\big|M^{\eta}_t( \delta p_{\delta^2}(x_0-\ARG))\big|\geq \deltaone /20\right]
  \leq \frac{80\,\overline x_2\,\delta}{\deltaone }.
\end{equation}
We will need this bound later also with $\deltaone$ replaced by $\deltaone/2$, that is,
\begin{equation}
\label{E4.09}
   \P\left[\sup_{t\geq 0}\big|M^{\eta}_t( \delta p_{\delta^2}(x_0-\ARG))\big|\geq \deltaone /40\right]
  \leq \frac{160\,\overline x_2\,\delta}{\deltaone }.
\end{equation}

Now we need to bound
\begin{equation}
 \P\left[\sup_{t\geq 0}\left|\delta \int_0^t X^{1,\eta}_{s}\left(\frac{1}{2}\sdv\,  p_{\delta^2}(x_0-\ARG)\right)\,ds\right|\geq \deltaone /10\right].
\end{equation}

Recall \equ{E3.12}. By following the proofs of Theorems~5.1, 5.2 in~\cite{bib:wal86}, one easily gets that if  $X^{1,\eta}$ solves \equ{E3.12}, then it also solves  the so-called mild form of the equation:
\begin{equation}
\label{E4.11}
X^{1,\eta}_t(\phi)=X^{1,\eta}_0(S_t\phi) + \int_0^t\int_{\R}S_{t-r}\phi(x) M^{\eta}(dr,dx)\mfa
\phi\in  C_b(\R).
\end{equation}
In fact, one can also derive \eqref{E4.11} directly using \eqref{E1.16} with $S_{t-s}\phi$ instead of $\phi_s\,,$ for $  s\leq t$,
and also using  \eqref{E1.21} together with the definition of $M^{\eta}(dr,dx)$.

Then from~\eqref{E4.11} we get
\begin{eqnarray}
\label{E4.12}
\int_0^t X^{1,\eta}_{s}\left(\frac{1}{2}\sdv  p_{\delta^2}(x_0-\ARG)\right)\,ds&=&
\int_0^t X^{1,\eta}_{0}\left( S_s\left(\frac{1}{2} \sdv  p_{\delta^2}(x_0-\ARG)\right)\right) \,ds\\
\nonumber
&&\mbox{}
+ \int_0^t \int_0^s\int_{\R}S_{s-r}\left(\frac{1}{2}\sdv  p_{\delta^2}(x_0-\ARG)\right)(x)\, M^{\eta}(dr,dx)
 \,ds\\
 \nonumber
&=:& I^{1,\eta}_t + I^{2,\eta}_t.
\end{eqnarray}
Let us take care of
 $I^{1,\eta}_t$ which is an easy term. Recall that $X^{1,\eta}_{0}=X^{1}_{0}$.
By using this, the Chapman-Kolmogorov equation, Fubini's theorem and properties of the heat semigroup, we easily get
$$\begin{aligned}
\int_0^t X^{1,\eta}_{0}\left( S_s\left(\frac{1}{2} \sdv  p_{\delta^2}(x_0-\ARG)\right)\right) \,ds
&=
X^{1}_{0}\left(  \int_0^t \partial_s p_{s+\delta^2}(x_0-\ARG)\,ds \right) \\
&= S_{t+\delta^2} X^{1}_{0}(x_0)-S_{\delta^2} X^{1}_{0}(x_0).
\end{aligned}$$
Then we immediately get
\begin{equation}
\left| \int_0^t X^{1,\eta}_{0}\left( S_s\left(\frac{1}{2} \sdv  p_{\delta^2}(x_0-\ARG)\right)\right) \,ds\right|
\leq |S_{t+\delta^2} X^{1}_{0}(x_0)|+|S_{\delta^2} X^{1}_{0}(x_0)|\leq
 2 \,\overline x_1\mfa x_0\in \R, t\geq 0.
 \end{equation}
Since $\delta<\frac{\deltaone }{40\,\overline x_1}$, we get
\begin{equation}
\label{E4.14}
\left|\delta I^{1,\eta}_t\right|< \deltaone /20\mfa t\geq0.
 \end{equation}

Now let us take care of $I^{2,\eta}_t$.
Recall $M^\eta$ from Lemma~\ref{L3.03}. Since $S_{s-r}\partial_x^2=\partial_x^2S_{s-r}$, using the Chapman-Kolmogorov equation, we get
\begin{equation}
\label{E4.15}
I^{2,\eta}_t=  \int_0^t \int_0^s\int_{\R} \frac{1}{2}\sdv p_{s-r+\delta^2}(x_0-x)M^{\eta}(dr,dx)
 \,ds.
 \end{equation}
Note that $\{(r,s)\mapsto \frac{1}{2}\sdv p_{s-r+\delta^2}(x_0-x), 0\leq r\leq s\leq t\}$ is bounded and continuous. Hence
we can apply Fubini's theorem (see~\eqref{eq:18_12_1}--\eqref{E3.11} for necessary bounds)
to get
\begin{equation}\begin{aligned}
\label{E4.16}
I^{2,\eta}_t&=  \int_0^t\int_{\R} \int_r^t \frac{1}{2}\sdv p_{s-r+\delta^2}(x_0-x)\,ds\; M^{\eta}(dr,dx)\\
&= \int_0^t \int_{\R} \big[p_{t-r+\delta^2}(x_0-x)-p_{\delta^2}(x_0-x)\big]\;  M^{\eta}(dr,dx)\\
&=  \int_0^t \int_{\R} p_{t-r+\delta^2}(x_0-x)\;  M^{\eta}(dr,dx)
- \int_0^t \int_{\R}p_{\delta^2}(x_0-x)\;  M^{\eta}(dr,dx).
\end{aligned}\end{equation}

Now we express the integral with respect to $M^\eta$ in terms of $X^{2,\eta}$ which is easier to handle since it does not move. By partial integration,
for $r\leq t$, we get
\begin{equation}\begin{aligned}
S_{t+\delta^2-r}\,X^{2,\eta}_r(x_0) &= S_{t+\delta^2}\,X^{2,\eta}_0(x_0) -
   \int_0^r \int_{\R} \partial_t\,p_{t+\delta^2-u}(x_0-x)\,  X^{2,\eta}_u(dx)\,du\\
 &\quad
  -\int_0^r \int_{\R} p_{t+\delta^2-u}(x_0-x)\,  M^{\eta}(du,dx).
\end{aligned}\end{equation}
Hence, for $r=t$ we get
\begin{equation}\begin{aligned}
\label{E4.18}
\int_0^t \int_{\R} p_{t+\delta^2-u}(x_0-x)\,  M^{\eta}(du,dx) &=
 -S_{\delta^2}X^{2,\eta}_t(x_0) + S_{t+\delta^2}X^{2,\eta}_0(x_0) \\
&\quad-
   \int_0^t \int_{\R} \partial_t\,p_{t+\delta^2-u}(x_0-x)\,  X^{2,\eta}_u(dx)\,du,
\end{aligned}\end{equation}
and
\begin{equation}
\label{E4.19}
\begin{aligned}
\sup_{t\geq 0}\left| \int_0^t \int_{\R} p_{t+\delta^2-u}(x_0-x)\,  M^{\eta}(du,dx)\right| &\leq
\sup_{t\geq 0}\left| S_{\delta^2}X^{2,\eta}_t(x_0)\right| +
\sup_{t\geq 0}\left| S_{t+\delta^2}X^{2,\eta}_0(x_0)\right| \\
&\quad+\sup_{t\geq 0}\left|
   \int_0^t \int_{\R} \partial_t p_{t+\delta^2-u}(x_0-x)\,  X^{2,\eta}_u(dx)\,du\right|\\[2mm]
&=: J^{1,\eta}+ J^{2,\eta}+ J^{3,\eta}.
\end{aligned}\end{equation}
Let us treat $J^{3,\eta}$. Fix $\alpha\in (0,1]$. Then decompose
\begin{equation}\begin{aligned}
 \int_0^t \int_{\R} \partial_t\, p_{t+\delta^2-u}(x_0-x)\,  X^{2,\eta}_u(dx)\,du&=
 \int_0^{t-\delta^{2\alpha}} \int_{\R} \partial_t p_{t+\delta^2-u}(x_0-x)\,  X^{2,\eta}_u(dx)\,du\\
 &\quad+
 \int_{t-\delta^{2\alpha}}^t \int_{\R} \partial_t p_{t+\delta^2-u}(x_0-x)\,  X^{2,\eta}_u(dx)\,du.
\end{aligned}\end{equation}
Note that for $t>0$ and $x\in \R$, we have
\begin{equation}
 \left|\partial_t p_{t}(x)\right| \leq\sqrt{2}\, t^{-1}\, p_{2t}(x).
\end{equation}
Then we get
\begin{equation}
\begin{aligned}
\left|\int_{t-\delta^{2\alpha}}^t \int_{\R} \partial_t p_{t+\delta^2-u}(x_0-x)\,  X^{2,\eta}_u(dx)\,du
 \right|\hspace{-3cm}\\
 &\leq \sqrt{1/\pi}\left|\int_{t-\delta^{2\alpha}}^t (t+\delta^2-u)^{-3/2}
  \int_{\R} \exp\left(-\frac{(x_0-x)^2}{4(\delta^2+\delta^{2\alpha})}\right)  X^{2,\eta}_u(dx)\,du
 \right|
 \\
 &\leq \sqrt{1/\pi}  \sup_{u\leq t}\left\{ \int_{\R} \exp\left(-\frac{(x_0-x)^2}{8\delta^{2\alpha}}\right)  X^{2,\eta}_u(dx)\right\}
 \int_{t-\delta^{2\alpha}}^t (t+\delta^2-u)^{-3/2}
  \,du
 \\
 &\leq  2(\sqrt{\pi}\delta)^{-1} \sup_{u\leq t}\left\{ \int_{\R} \exp(\left(-\frac{(x_0-x)^2}{8\delta^{2\alpha}}\right)  X^{2,\eta}_u(dx)\right\}
 \\
 &\leq 6\, \delta^{\alpha-1}\sup_{u\leq t}\left\{ S_{4\delta^{2\alpha}}  X^{2,\eta}_u(x_0)\right\}.
\end{aligned}\end{equation}
Also
\begin{equation}\begin{aligned}
 \left|\int_0^{t-\delta^{2\alpha}} \int_{\R} \partial_t p_{t+\delta^2-u}(x_0-x)  X^{2,\eta}_u(dx)\,du
 \right|\hspace{-3cm}
 \\
 &\leq \sqrt{1/\pi} \left|\int_0^{t-\delta^{2\alpha}} (t+\delta^2-u)^{-3/2}
  \int_{\R} \exp\left(-\frac{(x_0-x)^2}{4(t+ \delta^2)}\right)  X^{2,\eta}_u(dx)\,du
 \right|
 \\
 &\leq  \sup_{u\leq t}\left( X^{2,\eta}_u(\1)\right)
 (\delta^{2\alpha})^{-1/2}
 \\
 &\leq  \delta^{-\alpha}\sup_{u\leq t}\left( X^{2,\eta}_u(\1)\right).
\end{aligned}\end{equation}
Now take $\alpha=1/2$ and get
\begin{equation}
J^{3,\eta} \leq 7 \delta^{-1/2}\left(\sup_{t\geq 0} X^{2,\eta}_t(\1)+
  \sup_{t\geq 0} S_{4\delta}  X^{2,\eta}_t(x_0)\right).
\end{equation}
Recall that $X^{2,\eta}_t(\1)$ and $S_{4\delta}  X^{2,\eta}_t(x_0)$ are martingales. Hence, Doob's inequality gives
$$\P\left[7\delta^{1/2}\sup_{t\geq 0}X^{2,\eta}_t(\1)\geq \deltaone /240\right]\leq \frac{1680\,\delta^{1/2}}{\deltaone }X^{2,\eta}_0(\1)\leq1680\,\overline x_2\,\frac{\delta^{1/2}}{\deltaone}$$
and (using also \eqref{E4.07} with $\delta^2$ replaced by $4\delta$)
$$\P\left[7\,\delta^{1/2}\sup_{t\geq 0}S_{4\delta}  X^{2,\eta}_t(x_0)\geq \deltaone /240\right]\leq \frac{1680\,\delta^{1/2}}{\deltaone }\|S_{4\delta}X^{2,\eta}_0\|_\infty
\leq\frac{3360\,\overline x_2\,\delta^{1/2}}{\deltaone }
$$
since $\eta\leq\delta\leq\sqrt{4\delta}$.
Summing up, we have
$$\P\left[\delta J^{3,\eta}\geq\deltaone /120\right]\leq 5040 \,\overline x_2\,\frac{\delta^{1/2}}{\deltaone }.$$
Using Doob's inequality again for the martingale $S_{\delta^2}X^{2,\eta}_t(x_0)$ and using \eqref{E4.07}, we get
$$\P\left[\delta J^{1,\eta}\geq\deltaone /120\right]\leq \frac{240\,\overline x_2\,\delta}{\deltaone }.$$
Combining this with the estimate for $J^{3,\eta}$, we get
\begin{equation}
\label{E4.25}
 \P\left[\delta J^{1,\eta}\geq \deltaone /120\right]+\P\left[\delta J^{3,\eta}\geq \deltaone /120\right]
\leq 5280\,\overline x_2\,\frac{\sqrt\delta}{\deltaone }.
 \end{equation}
Now we bound $J^{2,\eta}$ using \eqref{E4.07} (recall that $\eta\leq\delta< \frac{\deltaone }{240\,\overline x_2}$):
$$J^{2,\eta}=\sup_{t\geq 0}S_tS_{\delta^2}X^{2,\eta}_0(x_0)\leq\sup_{t\geq 0}S_t\|S_{\delta^2}X^{2,\eta}_0(\ARG)\|_\infty\leq 2\,\overline x_2\leq \frac{\deltaone}{120}\,\delta^{-1}.$$

Combine  this, \eqref{E4.25} and \eqref{E4.19}
to get
\begin{equation}
\P\left[\sup_{t\geq 0}\left| \delta \int_0^t \int_{\R} p_{t+\delta^2-u}(x_0-x)  M^{\eta}(du,dx)\right|\geq \deltaone /40\right]
\leq 5280\,\overline x_2\,\frac{\sqrt\delta}{\deltaone }
 \end{equation}
 since $\eta\leq\delta< \frac{\deltaone }{240\, \overline x_2}$.

By this, \eqref{E4.16} and \eqref{E4.09}, we immediately get
\begin{equation}
\label{E4.27}
\P\left[\sup_{t\geq 0}\left|\delta I^{2,\eta}_t\right|\geq \deltaone /20\right]
\leq 5440\,\overline x_2\,\frac{\sqrt\delta}{\deltaone },
 \end{equation}
again since $\eta\leq\delta< \frac{\deltaone }{240\,\overline x_2}$.
This, \eqref{E4.14}  and \eqref{E4.12} imply that
\begin{equation}
\P\left[\sup_{t\geq 0}\left|\delta \int_0^t X^{1,\eta}_{s}\left(\frac{1}{2}\sdv  p_{\delta^2}(x_0-\ARG)\right)\,ds
 \right|\geq \deltaone /10\right]
\leq 5440\,\overline x_2\frac{\sqrt\delta}{\deltaone },
 \end{equation}
as we assumed $\eta\leq\delta\leq\deltaone / \max\{1, 240\,\overline x_2, 40\,\overline x_1\}$.

Combine this with \equ{E3.12}, \equ{E4.03} and \equ{E4.08}, and we are done.
\end{proof}

Define (recall that $\mathrm{supp}(X^2_0)=[0,1]$)
\begin{equation}
 \tau^\eta_i=\inf\{t\geq 0:\, \elleta{t}= i\eta\}.
\end{equation}
and for $s\geq0$ and $i=1,\ldots,\lceil\eta^{-1}\rceil$ the random variables
\begin{equation}
 Z^\eta_{i,s}:=\cases{X^{2,\eta}_{\tau^\eta_i+s}(\{i\eta\})-X^{2,\eta}_{\tau^\eta_i}(\{i\eta\}),&\mfalls \tau^\eta_i<\infty,\\0,&\msonst.}
\end{equation}

\begin{lemma}
\label{L4.02}
 For any $\deltaone \in(0,1)$ and all $\eta\leq\delta\leq\big(\deltaone /[4\max\{1,\,120\,\overline x_2,\,20\,\overline x_1,X^1_0(\1)+X^2_0(\1)\}])^4$,
we have
\begin{equation}
 \P\left[\sup_{i=1,\ldots,\lceil\eta^{-1}\rceil} \sup_{s\in[0,\delta]}Z^\eta_{i,s}\geq \deltaone \right]\leq
  30000\,\frac{\overline x_2\,\delta^{1/8}}{\deltaone ^{2}}
\end{equation}
\end{lemma}
\begin{proof}
We use the trivial estimate
\begin{equation}
\label{E4.32}
\P\left[\sup_{i=1,\ldots,\lceil\eta^{-1}\rceil} \sup_{s\in[0,\delta]}Z^\eta_{i,s}\geq \deltaone \right]
\leq
  \sum_{i=1}^{\lceil\eta^{-1}\rceil}
  \P\left[\sup_{s\in[0,\delta]} Z^\eta_{i,s}\geq \deltaone \right].
\end{equation}
For every $i=1,\ldots,\lceil\eta^{-1}\rceil$, define
$$\begin{aligned}
 A_{i,\delta,\deltaone ,\eta}&:= \left\{\sup_{t\geq 0} X^{1,\eta}_t([i\eta-\delta^{1/4},i\eta])\geq \deltaone /2
 \right\},\\
A^{\tau}_{i,\delta,\deltaone ,\eta}&:= \left\{ X^{1,\eta}_{\tau^\eta_i}([i\eta-\delta^{1/4},i\eta])\geq \deltaone /2
 \right\}.
\end{aligned}$$
Fix an arbitrary $i\in \{1,\ldots,\lceil\eta^{-1}\rceil\}.$
By Proposition~\ref{P4.01}, for $\eta\leq\delta\leq\big(\deltaone /[2\max\{1,\,240\,\overline x_2,\,40\,\overline x_1\}]\big)^4$, we get
\begin{equation}
\label{E4.33}
\P\left[A^{\tau}_{i,\delta,\deltaone ,\eta}\right]\leq \P\left[A_{i,\delta,\deltaone ,\eta}\right]\leq 11040\,\frac{\delta^{1/8}}{\deltaone }.
\end{equation}
Let us get the bound on $\sup_{s\in[0,\delta]}  X^{2,\eta}_{\tau^\eta_{i}+s}(\{i\eta\})-
X^{2,\eta}_{\tau^\eta_{i}}(\{i\eta\})$ on the event $(A^{\tau}_{i,\delta,\deltaone ,\eta})^{c}$. Recall that definition of the transition density $p^z$ of the heat flow $S^z$ with killing at $z\in\R$ from \equ{E1.15}. We start with an elementary observation that rephrases the reflection principle.
For $x<z$, we have
$$\begin{aligned}
\int_{-\infty}^z&p_\delta^z(x,y)\,dy+2\int_z^\infty p_\delta(y-x)\,dy\\
&=\int_{-\infty}^0p^z_\delta(x,y+z)\,dy+\int_{-\infty}^0p_\delta(z-y-x)\,dy+\int_0^\infty p_\delta(z+y-x)\,dy\\
&=\int_{-\infty}^0(p_\delta(y+z-x)-p_\delta(z-y-x))\,dy+\int_{-\infty}^0p_\delta(z-y-x)\,dy+\int_0^\infty p_\delta(z+y-x)\,dy\\
&=\int_{-\infty}^\infty p_\delta(y+z-x)\,dy\;=\;1.
\end{aligned}
$$
Hence, for any $z\in\R$ and any finite measure $\mu$ supported by $(-\infty,z)$ and $\varepsilon>0$, we have
\begin{equation}
\label{E4.34}
\begin{aligned}
\langle \mu,\1\rangle-\big\langle S^z_{\delta} \mu,\1\big\rangle
&\leq\mu((z-\varepsilon,z))+\langle \mu\1_{(-\infty,z-\varepsilon]},\1\rangle-\big\langle S^z_{\delta} (\mu\1_{(-\infty,z-\varepsilon]}),\1\big\rangle\\
&=\mu((z-\varepsilon,z))+2\langle S_{\delta} (\mu\1_{(-\infty,z-\varepsilon]}),\1_{(z,\infty)}\rangle\\
&=\mu((z-\varepsilon,z))+2\int_{(-\infty,z-\varepsilon]}\mu(dx)\big\langle S_{\delta} \delta_x,\1_{(z,\infty)}\big\rangle\\
&\leq\mu((z-\varepsilon,z))+2\mu(\R)\int_{[\varepsilon,\infty)}p_\delta(x)dx\\
&\leq\mu((z-\varepsilon,z))+\mu(\R)\sqrt{\delta}\varepsilon^{-1}e^{-\varepsilon^2/2\delta}\\
&\leq\mu((z-\varepsilon,z))+2\mu(\R)\delta^2\,\varepsilon^{-4}.
\end{aligned}
\end{equation}
Recall that $(X^{\eta,1}_t+X^{\eta,2}_t)(\1)$ is constant and that in the time interval $(\tau^\eta_i,\tau^\eta_{i+1})$ the process $X^{\eta,2}$ changes values only at $i\eta$, hence, $X^{\eta,1}_t(\1)+X^{\eta,2}_t(\{i\eta\})$ is constant in this time interval. We use \eqref{E4.34} with $\varepsilon=\delta^{1/4}$ and $z=i\eta$ to derive the last inequality in the following display formula on the event $\{\tau^\eta_i<\infty\}$
\begin{equation}
\begin{aligned}
\sup_{s\in[0,\delta]}  Z^\eta_{i,s}&=
-\left(
\inf_{s\in[0,\delta\wedge(\tau^\eta_{i+1}-\tau^\eta_i)]}
  \langle X^{1,\eta}_{\tau^\eta_{i}+s},\1\rangle -\langle X^{1,\eta}_{\tau^\eta_{i}},\1\rangle\right)
\\
&=
-\left(
\inf_{s\in[0,\delta\wedge(\tau^\eta_{i+1}-\tau^\eta_i)]}  \big\langle S^{i\eta}_s X^{1,\eta}_{\tau^\eta_{i}},\1\big\rangle -\big\langle
X^{1,\eta}_{\tau^\eta_{i}},\1\big\rangle\right)
\\
&\leq
 \big\langle X^{1,\eta}_{\tau^\eta_{i}},\1\big\rangle-\big\langle S^{i\eta}_{\delta} X^{1,\eta}_{\tau^\eta_{i}},\1\big\rangle
\\
&\leq
 X^{1,\eta}_{\tau^\eta_{i}}\big(\big(i\eta-\delta^{1/4},i\eta\big)\big)+2(X^1_0(\1)+X^2_0(\1))\,\delta.
\label{E4.35}
\end{aligned}
\end{equation}
Hence, for $\delta<\delta_0:=\deltaone /4(X^1_0(\1)+X^2_0(\1))$, we have
\begin{equation}
\label{E4.36}
 \P\left[\sup_{s\in[0,\delta]} Z^\eta_{i,s}\geq \deltaone ; (A^{\tau}_{i,\delta,\deltaone ,\eta})^{c}\right]
 = 0.
\end{equation}
Assume that $(\CF_t)_{t\geq0}$ is the filtration generated by $(X^1,X^2)$. By the optional stopping theorem, we get that $\big(X^{2,\eta}_{\tau^\eta_i+s}(\{i\eta\})\big)_{s\geq0}$ is a martingale with respect to the filtration $(\CF_{\tau^\eta_i+s})_{s\geq0}$. Hence, by Doob's inequality, we have (on the event $\{\tau^\eta_{i}<\infty\}$)
\begin{equation}
\label{E4.37}
 \P\left[\sup_{s\in[0,\delta]} Z^\eta_{i,s}\geq \deltaone  \Given
 \CF_{\tau^\eta_i}\right]
 \leq \frac{X^{2,\eta}_{\tau^\eta_i}(\{i\eta\})}{X^{2,\eta}_{\tau^\eta_i}(\{i\eta\})+\deltaone }\leq \frac{\eta\,\overline x_2}{\eta\,\overline x_2+\deltaone }\leq\frac{2\overline x_2\,\eta}{\deltaone },
\end{equation}
where the last inequality holds if $\eta\leq \deltaone /\overline x_2$. Together with \eqref{E4.33}, we get
\begin{equation}
\label{E4.38}
 \P\left[\sup_{s\in[0,\delta]}Z^\eta_{i,s}\geq \deltaone  \left|
 A^{\tau}_{i,\delta,\deltaone ,\eta}\right.\right]\,\P\left[A^{\tau}_{i,\delta,\deltaone ,\eta}\right]
 \leq 22080\,\frac{\overline x_2\,\eta}{\deltaone }\frac{\delta^{1/8}}{\deltaone }.
\end{equation}
From \eqref{E4.36} and \eqref{E4.38} we get
\begin{equation}
  \P\left[\sup_{s\in[0,\delta]} Z^\eta_{i,s}\geq \deltaone \right]\leq  22\,080\,\frac{\overline x_2\,\eta\,\delta^{1/8}}{\deltaone ^{2}},
\end{equation}
and the result follows by \eqref{E4.32} (note that $\eta\lfloor\eta^{-1}\rfloor\leq\frac{5}{4}<\frac{30\,000}{22\,080}$ by assumption).
\end{proof}

Since $X^{2,\eta}_{\tau^\eta_i}(\{i\eta\})\leq\overline x_2\eta<\xi/5$ (whence $(\xi-X^{2,\eta}_{\tau^\eta_i}(\{i\eta\}))^2\geq\frac{16}{25}\xi^2\geq\frac{3}{5}\xi^2$) by our choice of $\eta$ (on the event $\{\tau^\eta_i<\infty\}$), the following corollary is immediate.

\begin{corollary}
\label{C4.03}
For any $\deltaone \in(0,1)$ and all $\eta\leq\delta\leq\big(\deltaone /[8\max\{1,\,120\,\overline x_2,\,20\,\overline x_1,X^1_0(\1)+X^2_0(\1)\}]\big)^4$, we have
\begin{equation}
 \P\left[\sup_{i=1,\ldots,\lceil\eta^{-1}\rceil} \sup_{s\in[0,\delta]}\1_{\{\tau^\eta_i<\infty\}}X^{2,\eta}_{\tau^\eta_i+s}(\{i\eta\})\geq \deltaone \right]\leq
  50\,000\,\frac{\overline x_2\,\delta^{1/8}}{\deltaone ^{2}}.
\end{equation}
\end{corollary}

\begin{corollary}
\label{C4.05}
For any $\deltaone \in(0,1)$ and all $\eta\leq\delta\leq\big(\deltaone /[16\max\{1,\,120\,\overline x_2,\,20\,\overline x_1,(X^1_0(\1)+X^2_0(\1))\}]\big)^8$, we have
\begin{equation}
 \limsup_{\eta\downarrow 0} \P\left[\sup_{t\geq 0}\sup_{s\in[0,\delta]} X^{2,\eta}_{t+s}(\{\elleta{t+s}\})
 -X^{2,\eta}_{t}(\{\elleta{t}\})\geq \deltaone \right]\leq
  10^5\,\deltaone ^{-2}\delta^{1/16}.
\end{equation}
\end{corollary}
\begin{proof}
Recall that $\tau^\eta_1=0$ and define $\tau^\eta_{\lceil\eta^{-1}\rceil+1}=\infty$.
By splitting the interval $[0,\infty)$ into intervals $[\tau^\eta_i,\tau^\eta_{i+1}), i=1,\ldots,\lceil\eta^{-1}\rceil$, we get
\begin{equation}\begin{aligned}
\P\left[\sup_{t\geq 0}\sup_{s\in[0,\delta]} \left(X^{2,\eta}_{t+s}(\{\elleta{t+s}\})
 -X^{2,\eta}_{t}(\{\elleta{t}\})\right)\geq \deltaone \right]\hspace{-6cm}
\\
&=
  \P\left[\sup_{i=1,\ldots,\lceil\eta^{-1}\rceil}\sup_{r\in  [\tau^\eta_i,\tau^\eta_{i+1})} \sup_{(r-\delta)\wedge 0\leq t\le r} \left( X^{2,\eta}_{r}(\{\elleta{r}\})-X^{2,\eta}_{t}(\{\elleta{t}\})\right)\geq \deltaone \right]\leq I_1+I_2,
\end{aligned}\end{equation}
where
\begin{equation}
\label{E4.43}
I_1:=\P\left[\sup_{i=1,\ldots,\lceil\eta^{-1}\rceil}\sup_{r\in  [\tau^\eta_i,\tau^\eta_{i}+\delta+\delta^{1/2})}  X^{2,\eta}_{r}(\{i\eta\})\;\geq \deltaone \right]
\end{equation}
and
\begin{equation}
I_2:=\P\left[\sup_{\genfrac{}{}{0pt}{}{i=1,\ldots,\lceil\eta^{-1}\rceil}{ \tau^\eta_{i+1}-\tau^\eta_{i}>\delta+\delta^{1/2}}}\sup_{r\in  [\tau^\eta_i+\delta+\delta^{1/2},\tau^\eta_{i+1})}  \left( X^{2,\eta}_{r}(\{i\eta\})-X^{2,\eta}_{r-\delta}(\{i\eta\})\right)\geq \deltaone \right]
\end{equation}

By Corollary~\ref{C4.03}, we have the bound
\begin{equation}
I_1\leq  50\,000\,\frac{\overline x_2(2\delta)^{1/16}}{\deltaone ^{2}}\leq  10^5\,\frac{\overline x_2\,\delta^{1/16}}{\deltaone ^{2}}.
\end{equation}
For $I_2$, note that $r-\delta-\delta^{1/2}\in(\tau^\eta_i,\tau^\eta_{i+1})$, hence $$X^{1,\eta}_{r-\delta}=S^{i\eta}_{\delta^{1/2}}X^{1,\eta}_{r-\delta-\delta^{1/2}}$$
and thus
$$\|X^{1,\eta}_{r-\delta}\|_\infty\leq [X^1_0(\1)+X^2_0(\1)]\,\delta^{-1/4}.$$
Also note that for any $z\in\R$, we have
$$\big\langle S_\delta\1_{(-\infty,z]},\1_{[z,\infty)}\big\rangle\leq \delta^{1/2}.$$
Taking $z=i\eta$, we infer using also the reflection principle
\begin{equation}
\begin{aligned}
X^{2,\eta}_{r}(\{i\eta\})-X^{2,\eta}_{r-\delta}(\{i\eta\})
&= X^{1,\eta}_{r-\delta}(\1)-\big\langle S^{i\eta}_\delta X^{1,\eta}_{r-\delta},\1\big\rangle\\
&= 2\big\langle S_\delta X^{1,\eta}_{r-\delta},\1_{(i\eta,\infty)}\big\rangle\\
&\leq 2\|X^{1,\eta}_{r-\delta}\|_\infty\big\langle S_\delta \1_{(-\infty,i\eta)},\1_{(i\eta,\infty)}\big\rangle\\
&\leq\big[X^1_0(\1)+X^2_0(\1)\big]\,\delta^{1/4}<\deltaone .\\
\end{aligned}
\end{equation}
The last inequality is due to the assumption on $\delta$. Hence, $I_2=0$.
\end{proof}

From the above corollary, we immediately get
\begin{corollary}
\label{C4.06}
For any $\deltaone >0$,
\begin{equation}
 \lim_{\delta\downarrow 0}\limsup_{\eta\downarrow 0}
 \P\left[\sup_{t\geq 0}\sup_{s\in[0,\delta]}  X^{2,\eta}_{t+s}(\{\elleta{t+s}\})
 -X^{2,\eta}_{t}(\{\elleta{t}\})\geq \deltaone \right]=0.
\end{equation}
\end{corollary}
This implies that the limiting process $X^2$ (if exists) does not have jumps up (and $X^1$ jumps down).

\subsection{Tightness of $(X^{1,\eta}, X^{2,\eta})$ and $\elleta{}$}
\label{S4.2}
This section is devoted to the proof of tightness of $(X^{1,\eta}, X^{2,\eta})$ and $\elleta{}$.
For the rest of this section, we fix an arbitrary sequence
$(\eta_k)_{k=1,2,\ldots}$ such that $\eta_k\downarrow 0$.
With some abuse of notation denote the corresponding processes by $(X^{1,k}, X^{2,k})$ and $\ellk{}$ and define $(\cF^k_t)_{t\geq 0}$ as the filtration generated by $(X^{1,k}, X^{2,k})$.

Recall from \equ{E3.02} that
\begin{equation}
\label{E4.48}
 X^{2,k}_t(dx)= X^{2,k}_t(\{\ellk{t}\})\delta_{\ellk{t}}(dx)+\1_{(\ellk{t},\infty)}(x)X^{2,k}_0(dx).
\end{equation}
For simplicity, denote
\begin{equation}
 Y^k_t:= X^{2,k}_t(\{\ellk{t}\}).
\end{equation}

In fact, we will show a bit more than tightness of $(X^{1,\eta}, X^{2,\eta})$. We are going to prove the
following proposition.

\begin{proposition}
\label{P4.07}
 $\{(X^{1,k},Y^k,\ellk{})\}_{k\geq 1}$ is tight in $D_{M_F\times \R_+\times \R_+}$.  Moreover, $(\ellk{})_{k\geq1}$ is
$C$-tight in $D_{\R}$.
\end{proposition}
We prove Proposition~\ref{P4.07} via  a series of lemmas.
We start with proving the  $C$-tightness of $(\ellk{})_{k\geq 1}$ by checking the Aldous criterion of tightness.
\begin{lemma}[Aldous criterion for $\{\ellk{}\}_{k\geq 1}$\hspace*{-0.28pt}]
\label{L4.08}
\hspace{-1pt}Let $(\widetilde\tau^k)_{k\geq 1}$ be an arbitrary sequence of finite $(\cF^k_t)_{t\geq0}$-stopping times. Then for any $\ep>0$,
\begin{equation}
\label{E4.50}
\lim_{\delta\downarrow 0} \limsup_{k\to \infty}\P\big[\big|\ellk{\widetilde\tau^k+\delta}-\ellk{\widetilde\tau^k}\big|\geq \ep\big]=0.
\end{equation}
Moreover, $\big(\sup_{t\geq 0} \ellk{t}\big)_{k\geq0}$ is a tight sequence of random variables.
\end{lemma}
\begin{proof}
By construction and Assumptions~\ref{A1.6},  $\ellk{}$ takes values in $[0,1+\eta_k]$,
and thus tightness of $(\sup_{t\geq 0} \ellk{t})$ is trivial.

Recall the definition of $W^{\eta_k}_i$ in the lines preceding \equ{E3.21}. The idea of the proof is the following. If $\ellk{s}$ makes a quick leap forward, then on the way it has to wake up many sleeping colonies of amounts $W^{\eta_k}_i$ in a short time. This is very unlikely, if one of the sleeping colonies is too large. On the other hand, it is unlikely that all the sleeping colonies that are leapt over are small. By Proposition~\ref{P3.09}, we do have control of the sizes of sleeping colonies only in $[a,1-a]$ (for any small $a>0$) since we have control on the density of $X^2_0$ only away from the boundaries of the interval $[0,1]$. This leads to a small technical twist in the following argument.
Fix an arbitrary $\ep>0$. Then choose $\ep'\in(0,\ep)$ arbitrarily small.

Define the events
\begin{equation}
A_{k,\ep,\ep'}=\left\{ \forall j\in \left(\left\lceil \frac{\ep/3}{\eta_k}\right\rceil , \left\lceil \frac{1-\ep/3}{\eta_k}\right\rceil\right):   \max_{i:\, i\eta_k\in (0, \ep'/3)} W^{\eta_k}_{j+i} \geq
(\ep')^2/9\right\}.
\end{equation}
Note that for any interval $I\subset[0,1]$ of length at least $\varepsilon$, for any $k$ sufficiently large such that $\eta_k<\ep/6$, there exists a $j$ such that $j\eta_k\in(\ep/3,1-\ep/3)$ and such that  $[j\eta_k, j\eta_k+\ep'/3] \subset I$.

Now let $\delta>0$ and apply the observation with $I=[\ellk{\widetilde\tau^k},\ellk{\widetilde\tau^k+\delta}]$. Hence, we have (using Lemma~\ref{L4.02} in the second inequality)
\begin{equation}
\label{E4.52}
\begin{aligned}
\P\left[|\ellk{\widetilde\tau^k+\delta}-\ellk{\widetilde\tau^k}|\geq \ep, A_{k,\ep,\ep'}\right]
&\leq \P\left[  \sup_{i=1,\ldots,\lceil \eta_k^{-1}\rceil} \sup_{s\in[0,\delta]}X^{2,k}_{\tau^{\eta_k}_i+s}(\{i\eta_k\})-X^{2,k}_{\tau^{\eta_k}_i}(\{i\eta_k\})\geq (\ep')^2/9 \right]
\\
&\leq
 C\frac{\overline x_2\,\delta^{1/8}}{(\ep')^{4}}\To{\delta\downarrow0}0
\end{aligned}\end{equation}
uniformly in $k$ large enough. Note that in Proposition~\ref{P3.09}, by Assumption~\ref{A1.6}(ii, iii), we have
$$c_a:=\inf\{\underline x_{2,a,\delta}:\,\delta<a\}\geq\frac{1}{2}\inf\{X^2_0(x):\,x\in[a/2,1-a/2]\}>0.$$

Hence, by Proposition~\ref{P3.09} (with $a=\ep/3$ and $\delta=\ep'/3$), there exists a $c=c(\ep, X^2_0)>0$ such that
\begin{equation}
\P\big[A_{k,\ep,\ep'}^c\big] \leq e^{-c/\ep'}.
\end{equation}
Thus we get
\begin{equation}
\label{E4.54}
\lim_{\delta\downarrow 0} \limsup_{k\to \infty}\P\big[\big|\ellk{\widetilde\tau^k+\delta}-\ellk{\widetilde\tau^k}\big|\geq \ep\big]\leq e^{-c/\ep'}.
\end{equation}
Since $\ep'\in(0,\ep)$ was arbitrarily small, we are done.
\end{proof}

\begin{corollary}
\label{C4.09}
$(\ellk{})_{k\geq 1}$ is $C$-tight.
\end{corollary}
\begin{proof}
From the previous lemma and Theorem~6.8 in \cite{bib:wal86} we get that $\{\ellk{}\}_{k\geq 1}$ is tight in $D_{\R}$. Moreover, we can easily see that
\begin{equation}
\sup_{t\geq  0}\Delta \ellk{t} \leq \eta_k\limk 0.
\end{equation}
Hence, $C$-tightness follows by \cite[Proposition VI.3.26]{JacodShiryaev2003}.
\end{proof}

Now we will verify Aldous' criterion for $(Y^k)_{k\geq 1}$.

\begin{lemma}[Aldous criterion for $(Y^k)$]
\label{L4.10}
 Let $\big(\widetilde\tau^k\big)_{k\geq 1}$ be an arbitrary sequence of finite $\cF^k$-stopping times. Then for any $\ep>0$,
\begin{equation}
\label{E4.56}
\lim_{\delta\downarrow 0} \limsup_{k\to\infty}\P\Big[\Big|X^{2,k}_{\widetilde\tau^k+\delta}
\big(\big\{\ellk{\widetilde\tau^k+\delta}\big\}\big)-
X^{2,k}_{\widetilde\tau^k}\big(\big\{\ellk{\widetilde\tau^k}\big\}\big)\Big|\geq \ep\Big]=0.
\end{equation}
Moreover, $\left(\sup_{t\geq 0} X^{2,k}_{t}
\big(\big\{\ellk{t}\big\}\big)\right)_{k=1,2,\ldots}$ is a tight sequence of random variables.
\end{lemma}

\begin{proof}
Since
$$\big\langle X^{1,k}_t,\1\big\rangle+\big\langle X^{2,k}_t, \1\big\rangle = \big\langle X^{1,k}_0,\1\big\rangle+\big\langle X^{2,k}_0, \1\big\rangle\mfa t\geq0,$$
we get that
$$\sup_{t\geq 0} X^{2,k}_{t}\big(\big\{\ellk{t}\big\}\big)\,\leq\,\big\langle X^{1,k}_0,\1\big\rangle+\big\langle X^{2,k}_0,\1\big\rangle,\quad k=1,2,\ldots$$
 is tight.

The main work is proving \eqref{E4.56}. By Corollary~\ref{C4.05}, we only have to show
\begin{equation}
\label{E4.57}
\lim_{\delta\downarrow 0} \limsup_{k\to\infty}\P\Big[X^{2,k}_{\widetilde\tau^k+\delta}
\big(\big\{\ellk{\widetilde\tau^k+\delta}\big\}\big)-
X^{2,k}_{\widetilde\tau^k}\big(\big\{\ellk{\widetilde\tau^k}\big\}\big)\leq -\ep\Big]=0.
\end{equation}
Note that $\delta'\mapsto X^{2,k}_{\widetilde\tau^k+\delta'}\big(\big\{\ellk{\widetilde\tau^k+\delta'}\big\}\big)$ is nondecreasing as long as $X^{2,k}$ does not jump down in which case we have that $X^{2,k}_{\widetilde\tau^k+\delta^*-}\big(\big\{\ellk{\widetilde\tau^k}\big\}\big)>0$
 jumps to $X^{2,k}_{\widetilde\tau^k+\delta^*}\big(\big\{\ellk{\widetilde\tau^k}\big\}\big)=0$ for some $\delta^*>0$ and then  $X^{2,k}$  stays zero at $\ellk{\widetilde\tau^k}$ after time $\widetilde\tau^k+\delta^*$. Hence we have $X^{2,k}_{\widetilde\tau^k+\delta}
\big(\big\{\ellk{\widetilde\tau^k+\delta}\big\}\big)\geq X^{2,k}_{\widetilde\tau^k}\big(\big\{\ellk{\widetilde\tau^k}\big\}\big)$ unless there exists a $\delta'\in(0,\delta]$ such that $X^{2,k}_{\widetilde\tau^k+\delta'}\big(\big\{\ellk{\widetilde\tau^k}\big\}\big)=0$ which in turn implies $X^{2,k}_{\widetilde\tau^k+\delta}\big(\big\{\ellk{\widetilde\tau^k}\big\}\big)=0$. Hence, it is sufficient to show that
\begin{equation}
\label{E4.58}
\lim_{\delta\downarrow 0} \limsup_{k\to\infty}\P\Big[
X^{2,k}_{\widetilde\tau^k}\big(\big\{\ellk{\widetilde\tau^k}\big\}\big)\geq \ep;
X^{2,k}_{\widetilde\tau^k+\delta}\big(\big\{\ellk{\widetilde\tau^k}\big\}\big)=0 \Big]=0.
\end{equation}
By the optional stopping theorem, $Z_s:=X^{2,k}_{\widetilde\tau^k+s}\big(\big\{\ellk{\widetilde\tau^k}\big\}\big)$, $s\geq0$, is a martingale. On the event $Z_0\geq\varepsilon$ it takes values in $\{0\}\cup[\varepsilon,\infty)$. Hence, by Corollary \ref{C4.05}, for any $\deltaone>0$ and for any $\delta>0$ sufficiently small and for some constant $C$ that is independent of ${\delta}$ and $\deltaone$, we have for all $\eta>0$ small enough
$$\begin{aligned}
\P[Z_0\geq \varepsilon;\,Z_{\delta}=0]\;\leq\; \varepsilon^{-1}\E[(Z_{\delta}-Z_0)^+]\;&\leq\; \deltaone+\int_{\deltaone}^\infty\P[Z_{\delta}-Z_0\geq\deltaone' ]\,d\deltaone\\
&\leq\;\deltaone+10^5\,\delta^{\,1/16}\int_{\deltaone}^\infty\deltaone^{-2}\,d\deltaone\\
 \;&\leq\; \deltaone+10^5\,{\delta}^{\,1/16}/\deltaone.
\end{aligned}
$$
Letting first ${\delta}\to0$ and then $\deltaone\to0$, we get \eqref{E4.58}.
\end{proof}

We need a simple tightness criterion for finite measures on $\R$.
\begin{lemma}
\label{L4.11}
A family $\CF\subset M_F$ of finite measures on $\R$ is tight if
$$\sup_{\mu\in\CF}\int(1+x^2)\,\mu(dx)<\infty.$$
\end{lemma}
\begin{proof}
This is obvious.
\end{proof}
\begin{lemma}
\label{L4.12}
Let $\phi(x)=1+x^2$. For all $t\geq0$, we have almost surely
\begin{equation}
\big\langle X^{1,k}_t+X^{2,k}_t,\phi\big\rangle
\leq
\big\langle X^{1,k}_0+X^{2,k}_0,t+\phi\big\rangle<\infty.
\end{equation}
In particular, for any $T>0$, there is a compact set $K_T\subset M_F$ such that almost surely, $X^{1,k}_t,X^{2,k}_t\in K_T$ for all $t\in[0,T]$.
\end{lemma}
\begin{proof}
Let $(\psi_n)$ be a sequence in $C^2_b(\R)$ such that $\psi_n\geq0$, $\psi_n\uparrow\phi$ and $\psi''_n\leq2$ for all $n$. For example, let $\psi_n(x):=(1+x^2)/(1+x^2/n)$. By Lemma~\ref{L3.05}, we have
$$
\begin{aligned}
\big\langle X^{1,k}_t+X^{2,k}_t,\psi_n\big\rangle
&=\big\langle X^{1,k}_0+X^{2,k}_0,\psi_n\big\rangle
+\frac12\int_0^t\big\langle X^{1,k}_s,\psi_n''\big\rangle\,ds\\
&\leq\big\langle X^{1,k}_0+X^{2,k}_0,\phi\big\rangle
+\frac12\int_0^t\big\langle X^{1,k}_s+X^{2,k}_s,\psi_n''\big\rangle\,ds\\
&\leq\big\langle X^{1,k}_0+X^{2,k}_0,\phi\big\rangle
+\int_0^t\big\langle X^{1,k}_s+X^{2,k}_s,\1\big\rangle\,ds\\
&\leq\big\langle X^{1,k}_0+X^{2,k}_0,\phi\big\rangle
+t\,\big\langle X^{1,k}_0+X^{2,k}_0,\1\big\rangle.
\end{aligned}
$$
By monotone convergence, we infer
$$\big\langle X^{1,k}_t+X^{2,k}_t,\phi\big\rangle=\sup_{n\in\N}
\big\langle X^{1,k}_t+X^{2,k}_t,\psi_n\big\rangle\leq\big\langle X^{1,k}_0+X^{2,k}_0,t+\phi\big\rangle
.
$$
The second part of the claim follows by Lemma~\ref{L4.11}.
\end{proof}

The above lemmas almost immediately imply the tightness of $(Y^k,\ellk{})$ and of $X^{2,k}$:
\begin{lemma}
\label{L4.13}
\begin{itemize}
 \item[{\bf (a)}]
$\big((Y^k,\ellk{})\big)_{k\geq 1}$ is tight in $D_{\R^2}$;
\item[{\bf (b)}] $\big(X^{2,k}\big)_{k\geq 1}$ is tight in $D_{M_F}$.
\end{itemize}
\end{lemma}
\begin{proof}
{\bf (a)}
Note that $f_k\to f$ and $g_k\to g$ in $D_\R$ does not imply $f_k+g_k\to f+g$ or even $(f_k,g_k)\to (f,g)$ in $D_{\R^2}$. Some additional condition is needed.
We know that $(\ellk{})$ is $C$-tight in $D_{\R}$. So by \cite[Corollary VI.3.33]{JacodShiryaev2003}, to get the result,
 it is enough to show tightness of $(Y_k)$ in $D_{\R}$.
Since, by construction,
$$Y^k_t\leq \big\langle X^{1,k}_0,\1\big\rangle + \big\langle X^{2,k}_0,\1\big\rangle<\infty\mfa t\geq0,$$
the compact containment condition for $(Y^k)_{k\geq 1}$ is fulfilled. Now, Aldous' criterion for $(Y^k)_{k\geq 1}$
is satisfied by Lemma~\ref{L4.10}, and thus the result follows.
\paragraph{\bf (b)} Since
$$\big\langle X^{2,k}_t,\1\big\rangle\leq \big\langle X^{1,k}_0,\1\big\rangle + \big\langle X^{2,k}_0,\1\big\rangle<\infty\mfa t\geq0,$$
$(X^{2,k}_t)_{k\geq1}$ is a tight sequence of measures on the compact $[0,1+\max_k\eta_k]$. (In fact, the support of $X^{2,k}_t$ is contained in $[0,\eta_k\lceil\eta_k^{-1}\rceil]\subset[0,1+\eta_k]$.)

By \eqref{E4.48}, for any $\phi\in C_c(\R)$, we have
\begin{equation}
X^{2,k}_t(\phi)=Y^k_t \phi(\ellk{t}) + \int_{(\ellk{t},\infty)} \phi(x) X^{2,k}_0(dx).
\end{equation}
By tightness of $\big((Y^k,\ellk{})\big)_{k\geq 1}$  in $D_{\R^2}$ (with $(\ellk{})_{k\geq 1}$ being $C$-tight)
and by the very definition of $X^{2,k}_0$, we get that
\begin{equation}
\big(X^{2,k}(\phi)\big)_{k\geq 1}
\end{equation}
is tight in $D_{\R}$ for any $\phi\in C_c(\R)$.
Thus the result follows by Jakubowski's criterion of tightness of measure-valued processes (see \cite[Theorem 3.6.4]{Dawson1993} or \cite{Jakubowski1986}).
\end{proof}

\paragraph{Proof of Proposition~\ref{P4.07}.}
Now let us show that $\big(X^{1,k}\big)_{k\geq 1}$ is tight in $D_{M_F}$.  By Lemma~\ref{L4.13}(b), we
have that $\big(X^{2,k}\big)_{k\geq 1}$ is tight in $D_{M_F}$.  Therefore,
from \eqref{E3.13}, we obtain  that
for any $\phi \in  C_b^2(\R)\,,$  $\big(M^{\eta_k}(\phi)\big)_{k\geq 1}$ is tight in $D_{\R}$.  Since
$$X^{1,k}_t(\1)\leq X^{1}_0(\1)+X^{2}_0(\1)<\infty\mfa t\geq0,$$
and $\| \phi''\|_{\infty}<\infty$    we get that
$$\left(t\mapsto\int_0^{t} \Big\langle X^{1,k}_{s},\,\frac{1}{2}\phi''\Big\rangle\,ds\right)_{k\geq 1}$$ is uniformly Lipschitz continuous and is hence $C$-tight.
This, \eqref{E3.12} and Corollary~VI.3.33 in~\cite{JacodShiryaev2003} imply that $\{X^{1,k}(\phi)\}_{k\geq1}$ is tight in $D_{\R}$. Again, Jakubowski's criterion gives tightness of $\big( X^{1,k}\big)_{k\geq 1}$ in
$D_{M_F}$.

Let us  check that for any $\phi\in  C_b^2(\R)$
\begin{align}
\label{E4.62}
&\big(X^{1,k}(\phi)+Y^k\big)_{k\geq 1} \; \mbsl{is tight in } D_\R;\\
\label{E4.63}
&\big(X^{1,k}(\phi)+\ellk{}\big)_{k\geq 1}\; \hspace{5pt}\mbsl{is tight in } D_\R\,.
\end{align}
Note that $(X^{1,k}(\phi))_{k\geq1}$ is tight in $D_\R$, and $(\ellk{})_{k\geq 1}$ is $C$-tight in $D_\R$. Thus, by
Corollary~VI.3.33 in~\cite{JacodShiryaev2003}, \eqref{E4.63} follows.
Now use Corollary~\ref{C3.04} to get
\begin{equation}
\label{E4.64}\begin{aligned}
 X^{1,k}_t(\phi) + Y^k_t&= X^{1,k}_0(\phi)
+X^{2,k}_0((-\infty,\ellk{t}])
+\int_{0}^{t} \Big\langle  X^{1,k}_s\,,
\frac{1}{2}\phi''\Big\rangle\,ds \\
&\quad+
\int_{0}^{t}\int_{[0,\infty)} z (\phi(\ellk{s-})-1)
\, \CM^{k}_{\Delta}\big(ds,dz\big).
\end{aligned}\end{equation}
By tightness of $(X^{1,k}(\phi))_{k\geq 1}$ in $D_\R$, $C$-tightness of $(\ellk{})_{k\geq 1}$ and by the definition of $X^{2,k}_0$,
we get that
$$ X^{1,k}_0(\phi)
+X^{2,k}_0\big((-\infty,\ellk{t}]\big)
+\int_{0}^{t} \Big\langle  X^{1,k}_s\,,
\frac{1}{2}\phi''\Big\rangle\,ds$$
is $C$-tight. Tightness of $(X^{1,k}(\phi))_{k\geq 1}$ (for any $\phi\in  C_b^2(\R)$) immediately implies tightness
of the processes $$t\mapsto\int_{0}^{t}\int_{[0,\infty)} z \phi(\ellk{s-})
\, \CM^{k}_{\Delta}\big(ds,dz\big)$$ for any $\phi\in  C_b^2(\R)$; in particular it gives tightness of
$$t\mapsto \int_{0}^{t}\int_{[0,\infty)} z \big(\phi(\ellk{s-})-1\big)
\, \CM^{k}_{\Delta}\big(ds,dz\big).  $$
Again use Corollary~VI.3.33 in~\cite{JacodShiryaev2003} to get \eqref{E4.62}.

 Now use Lemma~\ref{L4.13}, tightness of $\{X^{1,k}\}_{k\geq 1}$, \eqref{E4.62}, \eqref {E4.63} and a simple adaptation of Problem~22 in Chapter~3 from \cite{EthierKurtz1986}, to derive the result.
\gdm

\section{Martingale problem for limit points, Proof of Theorem~\ref{T3}}
\setcounter{equation}{0}
\setcounter{theorem}{0}
\label{S5}
In the Section~\ref{S4} we proved tightness of $\{(X^{1,k},Y^{k},\ellk{})\}_{k\geq1}$. In this section, we show that the semimartingale characteristics converge (along a suitable subsequence $\eta_k\downarrow0$). We proceed by checking the conditions of Theorem IX.2.4 of \cite{JacodShiryaev2003}.
We start by proving some preparatory lemmas in Section~\ref{S5.1} and then finish the proof of Theorem~\ref{T3} in Section~\ref{S5.2}.

\subsection{Preparations}
\label{S5.1}

Let $(X^1,Y,\elll{}{})$ be a process whose law is an arbitrary limit point of the subsequence $(X^{1,k},Y^k,\ellk{})_{k=1,2,\ldots}$. Then there exists a subsequence
$\big((X^{1,k_m},Y^{k_m},\ellkm{})\big)_{m\geq 1}$ which converges in law to $(X^1,Y,\elll{}{})$. To simplify notation we will
continue to denote that subsequence by $((X^{1,k},Y^{k},\ellk{}))_{k\geq1}$. Moreover, by Skorohod's theorem, we may (and will) assume that the convergence of  $((X^{1,k},Y^{k},\ellk{}))_{k\geq1}$ to $(X^{1},Y,\elll{}{})$ holds a.s.
We will prove some helpful lemmas first.

When at time $t$, we have an atom of $X^2_t$ at the point $\elle{t}$, then this atom has started growing out of an infinitesimal mass of $X^2_t$ at $\elle{t}$ at some time $\tau(t)<t$. Furthermore, at some time $\sigma(t)>t$ the atom at $\elle{t}$ will collapse. Between $\tau(t)$ and $\sigma(t)$, the dormant frogs at $\elle{t}$ grow. Our aim is to describe the behaviour of our process by a decomposition of the time axis into intervals where the atom at $\elle{t}$ grows. These intervals may be very short and hence may be infinitely many (but countably many). Just as Brownian motion can be decomposed into excursions from $0$, we try a similar procedure here for the frog model. We formalize this by the following definitions.
Define
\begin{equation}
\label{E5.01}
\begin{aligned}
\tau(t)=\sup\{s\leq t:\,Y_s=0\}\mbu
\sigma(t)=\inf\{s\geq t:\,Y_s=0\}.
\end{aligned}\end{equation}

\begin{lemma}
\label{L5.1}
$\P$-a.s.{} for any $t\geq0$ such that $Y_t>0$, we have
\begin{equation}
\label{E5.02}
 X^1_s(x)=S^{\elle{t}}_{s-\tau(t)}X^1_{\tau(t)}(x)\mfa s\in\big(\tau(t),\sigma(t)\big).
\end{equation}

\end{lemma}
\begin{proof}
Note that $(Y^k_t-X^2_0(\{\ellk{t}\}))_{t\geq0}$ is a sequence of processes which have no jumps up and which are hence lower semicontinuous. Furthermore, $(Y^k_t-X^2_0(\{\ellk{t}\}))_{t\geq0}$ and $Y^k$ differ in supremum norm by at most $\overline x_2\eta_k\limk0$. Hence, $Y$ is the $D_\R$-limit of lower semicontinuous processes and is hence also lower semicontinuous.

Fix an arbitrary $t\geq0$ such that $Y_t>0$.
As $Y$ is lower semicontinuous, we have $\sigma(t)<t<\tau(t)$.

Let $\delta\in(0,(\sigma(t)-\tau(t))/2)$. Then $Y_s>0$ for all $s\in[\tau(t)+\delta/2,\sigma(t)-\delta/2]$. Since $Y$ is lower semicontinuous, we infer
$$\varepsilon:=\inf_{s\in[\tau(t)+\delta/2,\sigma(t)-\delta/2]}Y_s>0.$$
Hence,
$$\limsup_{k\to\infty}\inf_{s\in[\tau(t)+\delta,\sigma(t)-\delta]}Y^k_s\geq\varepsilon>0.$$
Note that $Y^k_s\leq\overline x_2\eta_k$ if $\Delta \ellk{s}>0$. Hence, $\ellk{s}$ is constant on $[\tau(t)+\delta,\sigma(t)-\delta]$ if $k$ is large enough such that $\overline x_2\eta_k<\varepsilon/2$ and $\inf_{s\in[\tau(t)+\delta,\sigma(t)-\delta]}Y^k_s\geq\varepsilon/2$. Hence, for those $k$, we have
\begin{equation}
\label{E5.03}
 X^{1,k}_s(x)= S^{\ellk{t}}_{s-(\tau(t)+\delta)}X^{1,k}_{\tau(t)+\delta}(x)\mf s\in\big(\tau(t)+\delta,\sigma(t)-\delta\big).
\end{equation}
By passing to the limit (and using uniform convergence of
$\ellk{}$ to $\elll{}{}$ and of $X^{1,k}$ to $X^1$ on compacts of $(\tau(t),\sigma(t))$)
we get
\begin{equation}
\label{E5.04}
 X^{1}_s(x)= S^{\elle{t}}_{s-(\tau(t)+\delta)}X^{1}_{\tau(t)+\delta}(x) \mf s\in\big(\tau(t)+\delta,\sigma(t)-\delta\big).
\end{equation}
By letting $\delta\downarrow 0$, we are done.
\end{proof}

As an immediate consequence of Lemma~\ref{L5.1} we get the following corollary.
\begin{corollary}
\label{C5.2}
$\P$-a.s. for any $\ep>0$,
$$
\1_{\{ Y^k_{t}\geq \ep\}}\roix X^{1,k}_{t}(\{\ellk{t}\})\limk
\1_{\{ Y_{t}\geq \ep\}}\roix X^{1}_{t}(\{\elle{t}\})\mfa t\geq 0.
$$
Moreover, if $\tau(t)=\sup\{s\leq t:\; Y_s=0\}$, then, $\P$-a.s. for any $\ep>0$,
\begin{equation}
\label{E5.05}
 \1_{\{ Y_{t}\geq \ep\}}X^{1}_{t}(x) =
\1_{\{ Y_{t}\geq \ep\}}S^{\elle{t}}_{t-\tau(t)} X^{1}_{\tau(t)}(x)\mfa x\in\R,\; t\geq 0.
\end{equation}
Also, $\P$-a.s. for any $\ep>0$,
\begin{equation}
\label{E5.06}
 \1_{\{ Y_{t}\geq \ep\}}Y_t =
\1_{\{ Y_{t}\geq \ep\}}\int_{\tau(t)}^t \roix X^1_s(\{\elle{s}\})\,ds
\mfa t\geq 0,
\end{equation}
or equivalently
\begin{equation}
\label{E5.07}
 \1_{\{ Y_{t}>0\}}\partial_tY_{t} = \1_{\{ Y_{t}>0 \}} \roix X^1_t(\{\elle{t}\})
\mfa t\geq 0,
\end{equation}
\end{corollary}
\bigskip

From \eqref{E1.23} and by construction we have
\begin{equation}
\label{E5.08}
 I( X^{k}_{t})=
\frac{\roix X^{1,k}_{t}(\{\ellk{t}\})}{Y^{k}_{t}}\1_{\{Y^k_t>0\}}\mf t\geq 0.
\end{equation}

\begin{lemma}
\label{L5.3}
Let  $G(x)=x^2\,\1_{(-1,1)}(x)$. Then there is a subsequence of $(\eta_k)$ which we also denote by $(\eta_k)$, such that $\P$-a.s., for all $T>0$,
\begin{equation}
\int_0^{t} I(X^k_{s-})G\big(-Y^k_{s-})\,ds
\rightarrow \int_0^{t} I( X_{s-})G\big(-Y_{s-})\,ds  \mas k\to\infty,
\end{equation}
uniformly on $t\in [0,T]$.
\end{lemma}
\begin{proof}
Recall that
$$\tau^k_i=\inf\big\{t:\;\ellk{t}=i\eta_k\big\}$$
and note that $\tau^k_1=0$ a.s.
Note that by \eqref{E1.20}, we have
$\roix X^{1,k}_s(\{i\eta\})=\partial_s X^{2,k}_s(\{i\eta\})$ for $s\in (\tau^k_{i},\tau^k_{i+1})$ and hence $\roix X^{1,k}_s(\{\ellk{s}\})=\partial_sY^k_{s}$.
We have
\begin{equation}
\label{E5.10}\begin{aligned}
 \int_0^{t} I(X^k_{s-})G\big(-Y^k_{s-})\,ds&=
\int_0^{t} \roix X^{1,k}_{s-}(\{\ellk{s-}\})Y^k_{s-} \1_{\{0<Y^k_{s-}<1\}} \,ds\\
&= \sum_{i=1}^{\lceil\eta_k^{-1}\rceil-1 }
\int_{\tau^k_i\wedge t}^{\tau^k_{i+1}\wedge t} \roix X^{1,k}_{s-}(\{\ellk{s-}\})Y^k_{s-} \1_{\{0<Y^k_{s-}<1\}} \,ds
\\
&= \sum_{i=1}^{\lceil\eta_k^{-1}\rceil-1 }
\int_{\tau^k_i\wedge t}^{\tau^k_{i+1}\wedge t} (\partial_s Y^k_{s-}) Y^k_{s-} \1_{\{0<Y^k_{s-}<1\}} \,ds
\\
&= \sum_{i=1}^{\eta_k^{-1} \ellk{t}-1}
 \frac{1}{2} \left((Y^k_{\tau^k_{i+1}-})^2\wedge 1)-
 (X^{2,k}_0(\{i\eta_k\})^2\wedge 1)\right)
\\
&\quad+ \frac{1}{2}\left(
((Y^k_{t-})^2\wedge 1) -(X^{2,k}_0(\{\ellk{t}\})^2\wedge 1)\right).
\end{aligned}\end{equation}
Note that
\begin{equation}
\label{E5.11}
\begin{aligned}
X^{2,k}_0(\{\ellk{t}\})^2+ \sum_{i=1}^{\eta_k^{-1} \ellk{t}-1}&
  X^{2,k}_0(\{i\eta_k\})^2\leq \sum_{i=1}^{\lceil\eta_k^{-1}\rceil }(\overline x_2\eta_k)^2\leq 2{\overline x_2}^2\,\eta_k\limk0.
\end{aligned}
\end{equation}
Also note that
\begin{equation}
\label{E5.12}
\left|((\Delta Y^k_{\tau^k_i})^2\wedge 1)-((Y^k_{\tau^k_i-})^2\wedge 1)\right|
\leq 2\, Y^k_{\tau^k_i-}X^{2,k}_0(\{\eta(i+1)\})+X^{2,k}_0(\{\eta(i+1)\})^2
\leq 2\,\overline x_2\,\eta_k Y^k_{\tau^k_i-}+{\overline x_2}^2\eta_k^2.
\end{equation}
Recall that $\P\Big[Y^k_{\tau^k_{i+1}-}>x\Given \tau^k_{i+1}<\infty\Big]=X^{2,k}_0(\{i\eta_k\})/(X^{2,k}_0(\{i\eta_k\})+x)$ for $x>0$.
Hence
$$\P\Big[Y^k_{\tau^k_{i+1}-}\1_{\{i\leq \eta_k^{-1}\ellk{t}-1\}}>x\Big]\leq \frac{X^{2,k}_0(\{i\eta_k\})}{X^{2,k}_0(\{i\eta_k\})+x}\leq \frac{\overline x_2\eta_k}{\overline x_2\eta_k+x}\leq \frac{\overline x_2\eta_k}{x}.$$
and
$$\begin{aligned}
\E\Big[(Y^k_{\tau^k_{i+1}-}\wedge K)\,\1_{\{i\leq \eta_k^{-1}\ellk{t}-1\}}\Big]
&\leq \int_0^K\frac{X^{2,k}_0(\{i\eta_k\})}{X^{2,k}_0(\{i\eta_k\})+x} dx\\
&\leq \int_0^K\frac{\overline x_2\eta_k}{\overline x_2\eta_k+x} dx\\
&= \overline x_2\eta_k\big(\log(K+\overline x_2\eta_k)-\log(\overline x_2\eta_k)\big).
\end{aligned}
$$
Note that $\ellk{t}\leq\sup\mathrm{supp} X^{2,k}_0\leq 2$. Hence, for any $\varepsilon>0$ and $K>0$, we have
$$\begin{aligned}
\P\left[\sum_{i=1}^{\eta_k^{-1}\ellk{t}}\eta_k Y^k_{\tau^k_i-}>\varepsilon\right]
&\leq\P\left[\max_{i=1,\ldots,\eta_k^{-1}\ellk{t}} Y^k_{\tau^k_i-}>K\right]
+\varepsilon^{-1}\,\E\left[\sum_{i=1}^{\eta_k^{-1}\ellk{t}}\eta_k (Y^k_{\tau^k_i-}\wedge K)\right]\\
&\leq2\overline x_2/K
+2\,\eta_k\varepsilon^{-1}\,\overline x_2\left(\log(K+\overline x_2\eta_k)+|\log(\overline x_2\eta_k)|\right).
\end{aligned}
$$
By letting first $k\to\infty$ and then $K\to\infty$, we get
\begin{equation}
\label{E5.13}
\limsup_{k\to\infty}\P\left[\sum_{i=1}^{\eta_k^{-1}\ellk{t}}\eta_k Y^k_{\tau^k_{i}-}>\varepsilon\right]=0\mbs{uniformly in}t\geq0
\end{equation}
and hence
\begin{equation}
\label{E5.14}
\limsup_{k\to\infty}\P\left[\sum_{i=1}^{\eta_k^{-1}\ellk{t}}\left|((\Delta Y^k_{\tau^k_{i+1}})^2\wedge 1)-((Y^k_{\tau^k_{i+1}-})^2\wedge 1)\right|>\varepsilon\right]=0\mbs{uniformly in}t\geq0.
\end{equation}

Combining \eqref{E5.10}, \eqref{E5.11} and \eqref{E5.14}, we get uniformly in $t\geq0$
 \begin{equation}
\label{E5.15}
\begin{aligned}
&\limsup_{k\to\infty}\P\left[\left|\int_0^{t} I(X^k_{s-})G\big(-Y^k_{s-})\,ds-\frac12\left(\left(\sum_{s<t:\,\Delta Y^k_s\neq0}
(\Delta Y^k_{s})^2\wedge 1\right)+((Y^k_{t-})^2\wedge 1)\right)\right|>\ve\right]\\
&\quad=\limsup_{k\to\infty}\P\left[\left|\int_0^{t} I(X^k_{s-})G\big(-Y^k_{s-})\,ds-\frac{1}{2}\left(
 \left(\sum_{i=1}^{\eta_k^{-1}\ellk{t}}(\Delta Y^k_{\tau^k_{i+1}})^2\wedge 1\right)+(Y^k_{t-})^2\wedge 1\right)\right|>\ve\right]=0.
\end{aligned}
\end{equation}

On the other hand define for the limiting process the sequence of time intervals $(\tau^Y_i, \sigma^Y_i)$
such that
\begin{eqnarray}
\nn
Y_{\tau^Y_i}&=&0\\
\nn
\Delta Y_{\sigma^Y_i}&=& - Y_{\sigma^Y_i-}\\
\nn
Y_t&>&0\mfa t\in (\tau^Y_i, \sigma^Y_i).
\end{eqnarray}
Note that there is at most a countable number of such intervals since all the intervals have positive lengths. We may call these intervals \emph{excursion intervals} of $Y$. For simplicity, we order these intervals in a way that
$$ \sigma^Y_{i+1}-\tau^Y_{i+1} \leq \sigma^Y_{i}-\tau^Y_{i}\mfa i\geq 1. $$

By Corollary~\ref{C5.2}, we know that
\begin{eqnarray}
\label{E5.16}
\roix X^{1}_{s}(\{\elle{s}\})= \partial_s Y_{s}\,,
\end{eqnarray}
for any $s$ such that $Y_s>0$. Hence, we get

\begin{equation}
\begin{aligned}
\label{E5.17}
 \int_0^{t} I( X_{s-})G\big(-Y_{s-})\,ds&=
\int_0^{t} \roix X^{1}_{s-}(\{\elle{s-}\})\,Y_{s-} \1_{\{0<Y_{s-}<1\}} \,ds\\
&= \sum_{i\geq 1:\,  \sigma^Y_i<t}
\int_{\tau^Y_i}^{\sigma^Y_{i}} \roix X^{1}_{s-}(\{\elle{s-}\})Y_{s-} \1_{\{0<Y_{s-}<1\}} \,ds
\\
&
+\int_{\tau(t)}^{t} \roix X^{1}_{s-}(\{\elle{s-}\})Y_{s-} \1_{\{0<Y_{s-}<1\}} \,ds\\
&= \sum_{i\geq 1:\,  \sigma^Y_i<t}
\int_{\tau^Y_i}^{\sigma^Y_{i}} (\partial_s Y_{s-}) Y_{s-} \1_{\{0<Y_{s-}<1\}} \,ds
\\
&
+\int_{\tau(t)}^{t} (\partial_s Y_{s-}) Y_{s-} \1_{\{0<Y_{s-}<1\}} \,ds\\
&=\frac{1}{2}\left( \left( \sum_{s<t:\, \Delta Y_s\neq0}
 (\Delta Y_s)^2\wedge1\right)+
((Y_{t-})^2\wedge 1)\right).
\end{aligned}
\end{equation}
Note that by the properties of the Skorohod topology
\begin{equation}
\label{E5.18}
\begin{aligned}
 \sum_{s<t:\, |\Delta Y_s|>\ve}
 \left((\Delta Y_s)^2\wedge1\right)
&\leq
\liminf_{k\to\infty}
\sum_{s<t:\, |\Delta Y^k_s|>\ve}
 \left((\Delta Y^k_s)^2\wedge1\right)\\
 &\leq
 \limsup_{k\to\infty}\sum_{s<t:\, |\Delta Y^k_s|>\ve}
 (\Delta Y^k_s)^2\wedge1
\leq
\sum_{s<t:\, |\Delta Y_s|\geq\ve}
 \left((\Delta Y_s)^2\wedge1\right).
\end{aligned}
\end{equation}
Note that, using the explicit distribution of $Y^k_{\tau^k_{i+1}-}$ given $\tau^k_{i+1}<\infty$, we get (for $k$ large enough such that $\overline x_2\eta_k<\varepsilon$)
$$
\E\Big[(\Delta Y^k_{\tau^k_{i+1}})^2\1_{\big\{|\Delta Y^k_{\tau^k_{i+1}}|\leq\varepsilon\big\}}\Given\tau^k_{i+1}<\infty\Big]
\leq \E\Big[(Y^k_{\tau^k_{i+1}-})^2\1_{\big\{Y^k_{\tau^k_{i+1}-}<2\varepsilon\big\}}\Given\tau^k_{i+1}<\infty\Big]
\leq 4\,\eta_k\,\overline x_2\,\varepsilon.
$$
This shows that (uniformly in $t\geq0$)
\begin{equation}
\label{E5.19}
 \limsup_{k\to\infty}\E\left[\sum_{s<t:\, 0<|\Delta Y^k_s|\leq\ve}
 (\Delta Y^k_s)^2\right]\leq 8\,\overline x_2\,\varepsilon.
\end{equation}
Putting together \eqref{E5.15}, \eqref{E5.17}, \eqref{E5.18} and \eqref{E5.19}, and passing to a suitable subsequence of $(\eta_k)$ if needed, we get the claim of the lemma.
\end{proof}
\begin{lemma}
\label{L5.4}
Let  $H\in \R\to\R_+$ be a bounded continuous function with compact support and such that $H$ equals $0$ in some neighbourhood of $0$. Then, $\P$-a.s., for all $T>0$,
\begin{equation}
\label{E5.20}
\int_0^{t} I(X^k_{s-})H\big(-Y^k_{s-})\,ds
\limk \int_0^{t} I( X_{s-})H\big(-Y_{s-})\,ds
\end{equation}
uniformly on $t\in [0,T]$.
\end{lemma}
\begin{proof}
The proof of this lemma goes along the lines of the proof of Lemma~\ref{L5.3} and it is in fact even simpler. We omit the details.
\end{proof}

\subsection{Convergence of functionals}
\label{S5.2}
Recall that $(X^1,Y,\elll{}{})$ is the almost sure limit of the subsequence $\big((X^{1,k},Y^{k},\ellk{})\big)_{k\geq1}$ in $D_{M_F\times \R\times \R}$. Let us assume that the subsequence $(\eta_k)$ is chosen such that also the claims of Lemma~\ref{L5.3} and \ref{L5.4} hold.

With a slight abuse of notation let  $\cN^k_{\Delta}$ denote a family of point processes on $\R_+\times\R$  related to the
jumps
  $(\Delta X^{2,k}_s)(\{\ellk{s-}\})$. Recall that this family introduced in \eqref{E3.06} was
 related to the jumps  $(\Delta X^{2,\eta}_s)(\{\elleta{s-}\})$  (see~\eqref{E3.05}) and was denoted by $\cN^{\eta}_{\Delta}$. That is,

\begin{equation}
 \cN^k_{\Delta}(dt, dz)=\sum_s\1_{\{(\Delta X^{2,k}_s)(\{\ellk{s-}\})\not= 0\}} \delta_{(s, -(\Delta X^{2,k}_s)(\{\ellk{s-}\})}(dt,dz).
\end{equation}

Let $\cN^{k,\prime}_{\Delta}$ be the corresponding compensator measure and let $\cM^{k}_{\Delta}=\cN^{k}_{\Delta}
  -\cN^{k,\prime}_{\Delta}$. Furthermore define $\cN_{\Delta}, \cN^{\prime}_{\Delta},\cM_{\Delta}$ similarly but with $Y^k$
replaced by $Y$.

Recall from
Corollary~\ref{C3.04} that $(X^{1,k},Y^k,\ellk{})$ is a solution to the following system of equations:
for any $\phi\in  C_b^2(\R)$
\begin{eqnarray}
\label{E5.22}
 \big\langle  X^{1,k}_t\,,\phi\big\rangle&=& \big\langle  X^{1}_0\,,\phi\big\rangle+\int_{0}^{t} \Big\langle  X^{1,k}_s\,,
\frac{1}{2}\phi''\Big\rangle\,ds +
\int_{0}^{t}\int_{[0,\infty)} z \phi(\ellk{s-})
\, \CM^k_{\Delta}\big(ds,dz\big)\\
 Y^k_t&=& X^{2}_0\big((-\infty,\ellk{t}]\big)-
\int_{0}^{t}\int_{[0,\infty)} z\,\CM^k_{\Delta}\big(ds,dz\big).
\end{eqnarray}
In fact, by the very defintion of $X^{1,k}_0$ and $X^{2,k}_0$, we have $X^{1,k}_0=X^{1}_0$ and $X^{2,k}_0((-\infty,i\eta])=X^2((-\infty,i\eta])$ for all $i\in\N_0$, hence $X^{2,k}_0\big((-\infty,\ellk{t}]\big)=X^{2}_0\big((-\infty,\ellk{t}]\big)$.

By Lemma~\ref{L3.01}, we have
\begin{equation}
\label{E5.24}
\cN^{k,\prime}_\Delta\big(dt,B\big)=
 \1_{B\setminus\{0\}} (Y^k_{t-})\,I(X^k_{t-})\,dt\,\mf B\subset \R_+\mbox{ measurable}.
\end{equation}
\begin{lemma}
\label{L5.5}
$(X^{1},Y,\elll{}{})$ is a weak solution to the following system of equations:
for any $\phi\in C_b^2(\R)$
\begin{eqnarray}
\label{E5.25}
 \big\langle  X^{1}_t\,,\phi\big\rangle&=& \langle  X^{1}_0\,,\phi\rangle+\int_{0}^{t} \Big\langle  X^{1}_s\,,
\frac{1}{2}\phi''\Big\rangle\,ds +
\int_{0}^{t}\int_{[0,\infty)} z \phi(\elle{s-})
\, \CM_{\Delta}\big(ds,dz\big)\\
\label{E5.26}
 Y_t&=& X^{2}_0\big((-\infty,\elle{t}]\big)-
\int_{0}^{t}\int_{[0,\infty)} z\,\CM_{\Delta}\big(ds,dz\big),
\end{eqnarray}
where
\begin{equation}
\CM_{\Delta}=\CN_{\Delta}-\CN^{\prime}_{\Delta},
\end{equation}
and $\CN_{\Delta}$ is the point process on $\R_+\times\R_+$ with the compensator process
\begin{equation}
\label{E5.28}
\cN^{\prime}_\Delta\big(dt,B\big)=
 \1_{B\setminus\{0\}} (Y_{t-})\,I( X_{t})\,dt\,\mf B\subset [0,\infty)\mbox{ measurable}.
\end{equation}
Moreover,
\begin{equation}
\label{E5.29}
\mathrm{supp}(X^1_t)\subset (-\infty, \elle{t}]\mfa t\geq0\mbs{such that}X^{2}_t(\1)>0.
\end{equation}

\end{lemma}

\begin{proof}
Clearly the point processes $ \cN^k_{\Delta}(dt, dz)$ and corresponding  compensator processes and martingale
measures could be extended to the processes  on $\R_+\times\R$ while giving zero mass to $\R_+\times (\R_-\setminus\{0\})$, that is, \mbox{$\cN^{k}_\Delta\big(dt,B\big)=\cN^{k,\prime}_\Delta\big(dt,B\big)=0$},  ${\rm for\ all}\ B\subset (\R_-\setminus\{0\})$\,. We will use this trivial extension throughout the proof.

First note that \eqref{E5.29} is obvious. Now we will derive \eqref{E5.26}.
Note that \begin{equation}
\label{E5.30}
\sum_{s\leq t}\big(\Delta X^{2,k}_{0}\big((-\infty,\ellk{s}]\big)\big)^2\limk0
\end{equation}
since the left-hand side of \eqref{E5.30} is bounded by
$$\sum_{i=1}^{\lceil\eta_k^{-1}\rceil}X^{2,k}_0(\{i\eta_k\})^2\leq \overline x_2\,\eta_k\,X^2_0(\1)\limk0.
$$
Hence, by Theorem IX.2.4 of \cite{JacodShiryaev2003},  it is enough to check that
\begin{equation}
\label{E5.31}
\Bigg(Y^k_t, X^{2}_0((-\infty,\ellk{t}]), \int_0^t\int_{\R}{\cN}_\Delta^{k,\prime}\big(ds,dz\big)\,G(z)\bigg)_{t\geq0}
\stackrel{k\to\infty}{\Longrightarrow}\bigg(Y_t, X^{2}_0((-\infty,\elle{t}]), \int_0^t\int_{\R}\cN_\Delta^{\,\prime}\big(ds,dz\big)\,G(z)\bigg)_{t\geq0}
\end{equation}
in $D_{\R_+\times\R_+\times \R}$ for
\begin{itemize}
\item[(i)]
each continuous $G\in C_b^+(\R)$ which is $0$ in some neighbourhood of $0$ and
\item[(ii)]
for some bounded and compactly supported function $G:\R\to\R$ that fulfills $G(x)=x^2$ in some neighbourhood of $0$. (In the notation of \cite{JacodShiryaev2003}, this $G$ is $h^2$, where $h$ is the truncation function used in IX.2.2 to define their $\tilde C$ and which is defined in II.2.3.)
\end{itemize}
Using that $(X^{1,k},Y^{k},\ellk{})\limk(X^{1},Y,\elll{}{})$ almost surely, we will derive that in fact the convergence~\eqref{E5.31} holds a.s.{} on that probability space.

Note that by \eqref{E5.24}, for $G$ satisfying either (i) or (ii), we have
\begin{equation}
\label{E5.32}
\int_{\R}G(z)\,\cN_{\Delta}^{k,\prime}\big(dt,dz\big)
=  I(X^k_{t-})G\big(Y^k_{t-})\,dt.
\end{equation}
Similarly, by \eqref{E5.28}, for $G$ satisfying either (i) or (ii), we have
\begin{equation}
\label{E5.33}
\int_{\R}G(z)\,\cN_{\Delta}^{\prime}\big(dt,dz\big)
=  I( X_{t-})G\big(Y_{t-})\,dt.
\end{equation}
Now the a.s.{} convergence
of $\int_0^{t}\int_{\R}{\cN}_\Delta^{k,\prime}\big(ds,dz\big)\,G(z)$
 follows from Lemmas~\ref{L5.3} and \ref{L5.4}. Recall that the convergence in these lemmas is uniform on compact time intervals and the limiting integrals are clearly continuous functions of $t$. In other words, the third entry on the left-hand side of \eqref{E5.31} is $C$-tight. Similarly, we get $C$-tightness and uniform convergence of $(X^{2,k}_0((-\infty,\ellk{t}]))_{t\geq0}=(X^2_0((-\infty,\ellk{t}]))_{t\geq0}$ from $C$-tightness of $(\ellk{t})_{t\geq0}$ and the fact that $X^2_0$ has a density.

By Corollary~VI.3.33 in~\cite{JacodShiryaev2003},
this together with the convergence of $(Y^k, X^{2,k}_0((-\infty, \ellk{t}])), k\geq 1,$   implies ``joint'' convergence  in~\eqref{E5.31}.  It is also straightforward to see from our construction that $\CM_{\Delta}\big(ds,B\big)=\cN_{\Delta}^{\prime}\big(dt,B\big)=\cN_{\Delta}\big(dt,B\big)=0$ for any $B\subset \R_-\,$. Hence, we get  \eqref{E5.26}.

Now let us show~\eqref{E5.25}. Using the above proof it is easy to get that for any $G\in  C_b(\R)$ such that $G(x)=0$ in a neighbourhood of zero,

\begin{equation}
\int_{0}^{t}\int_{[0,\infty)} z \phi(\ellk{s-}) G(z)
\, \CM^k_{\Delta}\big(ds,dz\big)\limk
\int_{0}^{t}\int_{[0,\infty)} z \phi(\elle{s-}) G(z)
\, \CM_{\Delta}\big(ds,dz\big)
\end{equation}
uniformly in $t$ on compacts.
In fact, there are only finitely many jumps of size $|z|\geq\varepsilon$ in any compact time interval and both the sizes and the positions converge due to the properties of convergence of $Y^{k}$ in the Skorohod space. The compensator of these large jumps converges almost surely by Lemma~\ref{L5.4}. Finally, we show that the martingale measure of the small jumps vanishes in $L^2$ as $\varepsilon\to0$ uniformly in $k$. This gives the desired almost sure convergence after passing to a subsequence of $(\eta_k)$ which we also denote by $(\eta_k)$.

$C$-tight.
In order to bound the small jumps in $L^2$, we use that by the Burkholder-Davis-Gundy inequality,  we have, for any $T>0$,
\begin{eqnarray}
&&
\nn
\hspace*{-5mm}\E\left[ \sup_{t\leq T}\left(
\int_{0}^{t}\int_{[0,\infty)} z \phi(\ellk{s-}) \1_{\{|z|\leq \ep\}} \, \CM^k_{\Delta}\big(ds,dz\big)\right)^2\right]+
\E\left[ \sup_{t\leq T}\left(
\int_{0}^{t}\int_{[0,\infty)} z \phi(\elle{s-}) \1_{\{|z|\leq \ep\}} \, \CM_{\Delta}\big(ds,dz\big)\right)^2\right]\\
\nn
&&\leq C\left\| \phi\right\|_{\infty}^2 \E\left[
\int_{0}^{T}\int_{[0,\infty)} z^2 \1_{\{|z|\leq \ep\}} \, \CN^{k,\prime}_{\Delta}\big(ds,dz\big)\right]+
 C\left\| \phi\right\|_{\infty}^2
\E\left(
\int_{0}^{T}\int_{[0,\infty)} z^2 \1_{\{|z|\leq \ep\}} \, \CN^{\prime}_{\Delta}\big(ds,dz\big)\right)
\end{eqnarray}
Arguing as in the proof of Lemma~\ref{L5.3}, derivation of \eqref{E5.19}, the right hand side is bounded by
$$C\,\|\phi\|_\infty^2\,\varepsilon\, \overline x_2.$$ This finishes the proof.
\end{proof}
Recall the Poisson point process $\cN$ and the corresponding martingale measure $\CM$ that were introduced in \eqref{E1.08} and \eqref{E1.09}. Now we are ready to state the proposition that
will help finishing the proof of Theorem~\ref{T3}.
\begin{proposition}
\label{P5.6}
Any limiting point $(X^{1},Y,\elll{}{})$ of $\{(X^{1,k},Y^k,\ellk{})\}_{k\geq 1}$ is a weak solution to the following system of equations:
for any $\phi\in C_b^2(\R)\,,$
\begin{align}
\label{E5.35}
 \langle  X^1_t\,,\phi\rangle&= \langle  X^1_0\,,\phi\rangle+\int_{0}^{t} \langle  X^1_s\,,
\frac{1}{2}\phi''\rangle\,ds +
\int_{0}^{t}\int_{[0,\infty)} Y_{s-}\phi(\elle{s-})
\1_{[0,I(X_{s-})]}(a)\,\CM\big(ds,da\big)\\
 Y_t&= X^2_0((-\infty,\elle{t}])-
\int_{0}^{t}\int_{[0,\infty)} Y_{s-}\1_{[0,I(X_{s-})]}(a)\,\CM\big(ds,da\big).
\end{align}

In addition, we have
$$ \elle{t}= \inf\big\{x:\; X^2_t((-\infty,x])>0\big\}\wedge1.$$
\end{proposition}
\begin{proof}
The proof of this proposition is rather standard and follows line by line the proof of Lemma~\ref{L3.05}.
\end{proof}

\paragraph{Proof of Theorem~\ref{T3}.}
We define
$$X^{2}_t:=Y_t\delta_{\elle{t}}+\1_{(\elle{t},\infty)}X^2_0.$$
This is consistent with (recall \eqref{E3.04})
$$X^{2,k}_t=Y^k_t\delta_{\ellk{t}}+\1_{(\ellk{t},\infty)}X^{2,k}_0.$$
With this definition, $X^2$ is clearly a continuous functional of $\elll{}{}$ and $Y$ and is hence the limit of $X^{2,k}$ as $k\to\infty$.

By Proposition \ref{P5.6}, the
process $(X^1,X^2)$ is a weak solution to \eqref{E1.24}, and thus the proof of Theorem~\ref{T3} is finished.
\gdm

\section*{Acknowledgment}
We would like to express our gratitude to two anonymous referees who carefully read our paper and made many very helpful suggestions.

\end{document}